\documentclass[11pt,a4aper]{article}
\pdfoutput=1

\usepackage[UKenglish]{isodate}
\cleanlookdateon

\usepackage{amsfonts,amsmath,amssymb,amsthm,mathtools,microtype} 
\usepackage[dvipsnames]{xcolor} 
\usepackage[noBBpl]{mathpazo} 

\usepackage[labelsep=period,labelfont=bf,justification=centering]{caption}
\usepackage{float,graphicx,subcaption}
\usepackage{booktabs,makecell}
\usepackage{enumitem,mdwlist}
\setlist[itemize]{topsep=0ex,itemsep=0ex,parsep=0.4ex}
\setlist[enumerate]{topsep=0ex,itemsep=0ex,parsep=0.4ex}

\usepackage{parskip,fullpage}
\usepackage{standalone}
\usepackage{thm-restate}
\usepackage[textsize=scriptsize]{todonotes}
\usepackage{comment}

\usepackage{algorithm2e}
\usepackage{array}
\newcommand{\PreserveBackslash}[1]{\let\temp=\\#1\let\\=\temp}
\newcolumntype{C}[1]{>{\PreserveBackslash\centering}p{#1}}

\usepackage[hyphens]{url} 
\usepackage[linktoc=all,hidelinks,colorlinks,unicode=true]{hyperref} 
\usepackage[capitalise,compress,nameinlink,noabbrev]{cleveref} 
\hypersetup{linkcolor={blue!70!black},citecolor={black},urlcolor={blue!70!black}}
\usepackage{hyperref}

\usepackage{tikz}
\usetikzlibrary{positioning}
\usetikzlibrary{arrows}
\usetikzlibrary{hobby}
\newcommand*{\braceme}[6][]{
\draw[
    shift={(#3:#2)},
    right to reversed-right to reversed,
    shorten >=-.75\pgflinewidth,
    #1
    ] (0,0)
        arc[radius=#2, start angle=#3, end angle=#3+(#4-#3)/2] node[
        above=2pt] (#5) {#6};
\draw[
    shift={({#3+(#4-#3)/2}:#2)},
    left to reversed-left to reversed,
    shorten <=-.75\pgflinewidth,
    #1
    ] (0,0)
        arc[radius=#2, start angle=#3+(#4-#3)/2, end angle=#4];
}

\newtheorem{theorem}{Theorem}
\newtheorem{lemma}{Lemma}[section]
\newtheorem{claim}[lemma]{Claim}
\newtheorem*{claim*}{Claim}
\newtheorem{conjecture}{Conjecture}

\theoremstyle{definition}
\newtheorem{remark}[lemma]{Remark}

\makeatletter
\def\namedlabel#1#2{\begingroup
   \def\@currentlabel{#2}%
   \label{#1}\endgroup
}
\makeatother

\newenvironment{poc}{\begin{proof}[Proof of
    Claim]}{\end{proof}}

\crefname{subsection}{Subsection}{Subsections}

\renewcommand{\ge}{\geqslant}
\renewcommand{\le}{\leqslant}
\renewcommand{\geq}{\geqslant}
\renewcommand{\leq}{\leqslant}

\newcommand*{\eps}{\varepsilon}

\newcommand*{\cB}{\mathcal{B}}

\newcommand*{\cA}{\mathcal{A}}

\newcommand{\colora}{Goldenrod}
\newcommand{\colorb}{SkyBlue}
\newcommand{\colorc}{Sepia}
\newcommand{\colord}{orange}
\newcommand{\colore}{MidnightBlue}
\newcommand{\colorf}{white}

\newcommand{\coloredbullet}[1]{\raisebox{-0.3ex}{\textcolor{#1}{\text{\Large \textbullet}}}}

\newcommand{\colb}{\coloredbullet{\colorb}}
\newcommand{\colc}{\coloredbullet{\colorc}}

\def\ray{5pt}
\def\dash{10pt}

\tikzset{
    dashone/.style={dash pattern=on 5pt off 15pt},
  }
  
\newcommand{\perform}[2]{\ensuremath{%
    {#1}^{\rightarrow}\left( #2 \right)
  }}

\title{A Recolouring Version of a Conjecture of Reed}

\date{\today}

\author{
Lucas De Meyer\footnotemark[1] \and 
Cl\'ement Legrand-Duchesne\footnotemark[2]\footnotemark[0] \and
Jared Le\'on\footnotemark[3] \and
Tim Planken\footnotemark[4] \and
Youri Tamitegama\footnotemark[5]
}

\begin{document}
\maketitle

\renewcommand{\thefootnote}{\fnsymbol{footnote}} 

\footnotetext[1]{Université Claude Bernard Lyon 1, CNRS, INSA Lyon, LIRIS, UMR5205, Villeurbanne France
(\textsf{\href{mailto:lucas.de-meyer@univ-lyon1.fr}{lucas.de-meyer@univ-lyon1.fr}}).}
\footnotetext[2]{Theoretical Computer Science Department, Faculty of Mathematics and Computer Science, Jagiellonian University, Kraków, Poland
(\textsf{\href{mailto:clement.legrand-duchesne@uj.edu.pl}{clement.legrand-duchesne@uj.edu.pl}}).}
\footnotetext[0]{CNRS, LaBRI, Université de Bordeaux, Bordeaux, France.}
\footnotetext[3]{Mathematics Institute, University of Warwick, United Kingdom
(\textsf{\href{mailto:jared.leon@warwick.ac.uk}{jared.leon@warwick.ac.uk}}).}
\footnotetext[4]{Institute of Discrete Mathematics and Algebra, Faculty of Mathematics and Computer Science, Technische Universität Bergakademie Freiberg, Germany
(\textsf{\href{mailto:tim.planken@math.tu-freiberg.de}{tim.planken@math.tu-freiberg.de}}).}
\footnotetext[5]{Mathematical Institute, University of Oxford, United Kingdom
(\textsf{\href{mailto:tamitegama@maths.ox.ac.uk}{tamitegama@maths.\allowbreak ox.ac.uk}}).}

\renewcommand{\thefootnote}{\arabic{footnote}} 
\begin{abstract}
  Reed conjectured that the chromatic number of any graph is closer to its
  clique number than to its maximum degree plus one. We consider a recolouring
  version of this conjecture, with respect to Kempe changes. Namely, we
  investigate the largest $\eps$ such that all graphs $G$ are $k$-recolourable
  for all $k \ge \lceil \eps \omega(G) + (1 -\eps)(\Delta(G)+1) \rceil$.\\
  For general graphs, an existing construction of a frozen colouring shows that
  $\eps \le 1/3$. We show that this construction is optimal in the sense that
  there are no frozen colourings below that threshold. For this reason, we
  conjecture that $\eps = 1/3$. For triangle-free graphs, we give a
  construction of frozen colourings that shows that $\eps \le 4/9$, and prove
  that it is also optimal.\\
  In the special case of odd-hole-free graphs, we show that $\eps = 1/2$, and
  that this is tight up to one colour.

\end{abstract}

\section{Introduction}

The chromatic number $\chi(G)$ of any graph $G$ lies between its clique number
$\omega(G)$ and the maximum degree $\Delta(G)$ plus one. A natural question is which of these two bounds
is closer to the chromatic number. Reed~\cite{MR1610746} conjectured that the
chromatic number of any graph~$G$ is at
most~$\left\lceil \left(\omega(G) + \Delta(G) +1\right) / 2\right\rceil$.

In the same article, Reed proved using the probabilistic method that this bound
is tight if true (see Theorem 2 in~\cite{MR1610746}). He also proved that there exists $\eps > 0$ such that for all
$G$, $\chi(G) \le \lceil \eps\omega(G) + (1-\eps)(\Delta(G) +1) \rceil$. More
recently, King and Reed~\cite{MR3431290} gave a significantly shorter proof of
this result and proved that for large enough $\Delta$, Reed's conjecture holds
for $\eps \le \frac{1}{320e^6}$. This conjecture generated a lot of interest over the past years and similar statements were proved in~\cite{MR1610746,zbMATH06827195} for increasing values of $\eps$.
The best bound known today is due to Hurley, de Joannis
de Verclos and Kang~\cite{MR4262443}, who proved that for all graphs $G$ with
sufficiently large maximum
degree,~$\chi(G) \le \lceil 0.119\omega(G) + 0.881(\Delta(G) + 1) \rceil$.

Reed's conjecture has been proven for several hereditary graph classes~\cite{king2009Clawfree,aravind2011Bounding,2016arXiv161102063W}.
Finally, list colouring versions and local strengthenings of Reed's conjecture were also considered in~\cite{zbMATH06827195,kelly2020Local} and~\cite{king2009Clawfree,aravind2011Bounding,kelly2020Local} respectively. More precisely, let $f(G) = \max_{v \in V(G)} \lceil (\omega(v) + \deg(v) +1)/2\rceil$, where $\omega(v)$ is the size of the largest clique containing $v$, the following local strengthening of Reed conjecture was introduced by King in~\cite{king2009Clawfree}:
\begin{conjecture}[Local Reed Conjecture]
    All graph $G$ are $f(G)$-colourable.
\end{conjecture}

\subsection{A recolouring version of Reed's conjecture}
In this article, we consider a recolouring variation of Reed's conjecture, that was
introduced by Bonamy, Kaiser and Legrand-Duchesne (see Section~2.3.1.2 in~\cite{legrand-duchesne2024Exploring}). Given a graph and a proper colouring of its vertex set, a \emph{Kempe chain} is
a connected bichromatic component of the graph. A \emph{Kempe change} consists in swapping the
two colours within a Kempe chain (see~\cref{fig:kempe}), thereby resulting in another proper colouring.
This reconfiguration operation was introduced in 1879 by Kempe in an attempt to
prove the Four-Colour theorem. Kempe changes are decisive in the existing proofs
of the Four-Colour theorem and of Vizing's edge-colouring theorem. We say that a
graph is \emph{$k$-recolourable} if all its $k$-colourings are \emph{Kempe equivalent}, that is,
if any $k$-colouring can be obtained from any other through a series of Kempe changes.
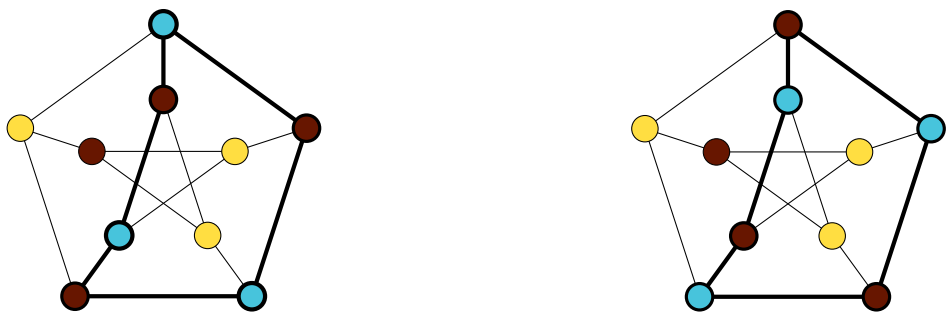
\begin{figure}[h!]
  \centering
  \hfill
  \begin{subfigure}[b]{0.45\textwidth}
    \centering
    \begin{tikzpicture}
      \foreach \x in {1,...,5}{
        \coordinate (a\x) at (90+\x*72:1cm);
        \coordinate (b\x) at (90+\x*72:2cm);
      }
      \foreach \x in {1,...,5}{
        \draw (a\x) -- (b\x);
        \pgfmathsetmacro{\y}{int(Mod(\x,5)+1)}
        \pgfmathsetmacro{\z}{int(Mod(\x+1,5)+1)}
        \draw (b\x) -- (b\y);
        \draw (a\x) -- (a\z);
      }

      \draw[ultra thick] (b5) -- (b4) -- (b3) -- (b2) -- (a2) -- (a5) -- (b5);

      \foreach \x in {b1,a3,a4}{
        \node[circle,draw=black,inner sep = 1pt, minimum size = 2*\ray,fill=\colora]  at (\x) {};
      }
      \node[circle,draw=black,inner sep = 1pt, minimum size = 2*\ray,fill=\colorc]  at (a1) {};
      \foreach \x in {b2,b4,a5}{
        \node[circle,draw=black,very thick,inner sep = 1pt, minimum size = 2*\ray,fill=\colorc]  at (\x) {};
      }
      \foreach \x in {b3,b5,a2}{
        \node[circle,draw=black,ultra thick,inner sep = 1pt, minimum size = 2*\ray,fill=\colorb]  at (\x) {};
      }

    \end{tikzpicture}
  \end{subfigure}
  \hfill
  \begin{subfigure}[b]{0.45\textwidth}
    \centering
    \begin{tikzpicture}
      \foreach \x in {1,...,5}{
        \coordinate (a\x) at (90+\x*72:1cm);
        \coordinate (b\x) at (90+\x*72:2cm);
      }
      \foreach \x in {1,...,5}{
        \draw (a\x) -- (b\x);
        \pgfmathsetmacro{\y}{int(Mod(\x,5)+1)}
        \pgfmathsetmacro{\z}{int(Mod(\x+1,5)+1)}
        \draw (b\x) -- (b\y);
        \draw (a\x) -- (a\z);
      }

      \draw[ultra thick] (b5) -- (b4) -- (b3) -- (b2) -- (a2) -- (a5) -- (b5);

      \foreach \x in {b1,a3,a4}{
        \node[circle,draw=black,inner sep = 1pt, minimum size = 2*\ray,fill=\colora]  at (\x) {};
      }
      \node[circle,draw=black,inner sep = 1pt, minimum size = 2*\ray,fill=\colorc]  at (a1) {};
      \foreach \x in {b2,b4,a5}{
        \node[circle,draw=black,very thick, inner sep = 1pt, minimum size = 2*\ray,fill=\colorb]  at (\x) {};
      }
      \foreach \x in {b3,b5,a2}{
        \node[circle,draw=black,very thick, inner sep =1pt, minimum size=  2*\ray,fill=\colorc]  at (\x) {};
      }
    \end{tikzpicture}
  \end{subfigure}
  \caption{Two 3-colourings of the Petersen graph that differ by one Kempe
    change. The corresponding $\{\colb,\colc\}$-Kempe chain is thickened.}
  \label{fig:kempe}
\end{figure}

The most common obstruction for $k$-recolourability is the existence of a
\emph{frozen $k$-colouring}, a $k$-colouring in which any two colours span a
connected subgraph. As a result, the partition induced by the colour classes is
invariant under Kempe changes, and if this colouring is not unique (up to colour
permutation), then the
graph is not $k$-recolourable.  Besides frozen colourings, the only other known
obstruction to recolourability is a topological argument of Mohar and
Salas~\cite{mohar2009New,mohar2010Nonergodicity} that is specific to
3-colourings of highly regular planar or toroidal graphs.

Bonamy, Kaiser and Legrand-Duchesne asked for the largest~$\eps$ such that all graphs $G$
are $k$-recolourable for all
$k \ge \lceil \eps\omega(G) + (1-\eps)(\Delta(G) +1) \rceil$. It remains open
whether this holds for some positive $\eps$. On the other hand, the parameter
$\eps$ is known to be at most $1/3$. Indeed, Bonamy, Heinrich, Legrand-Duchesne,
and Narboni~\cite{MR4660623} constructed a random graph with degrees
concentrating around $3n/4$ and expected clique number $\Theta(\log(n))$,
that admits a frozen $n/2$-colouring and is not $n/2$-recolourable with high
probability. In particular, this hints that Kempe changes are unlikely to be of any use to prove
Reed's conjecture in the range ${\eps \in [1/3,1/2]}$.

In \cref{sec:tight}, we prove that this construction is optimal in general:
\begin{restatable}{theorem}{tight}\label{thm:no_frozen}
  Let $G$ be a connected graph.
  For any $\eta \le 1/3$, if $G$ admits a frozen $k$-colouring
  for
  $k \le \lceil\eta\omega(G) + (1-\eta)(\Delta(G) +1)\rceil$, then it is unique up to permuting colours.
\end{restatable}

Since the random construction of Bonamy, Heinrich, Legrand-Duchesne and Narboni admits large cliques of size $\Theta(\log k)$, it is natural to ask whether similar obstructions occur in graphs of bounded clique number.
We fully resolve this question for triangle-free graphs by giving an optimal construction of a triangle-free graph with a frozen non-unique $k$-colouring and
maximum degree at most $(9k-1)/5$:

\begin{restatable}{theorem}{trianglefree}\label{thm:triangle_free}
\begin{enumerate}[label =(\roman*), ref=\cref{thm:triangle_free}.(\roman*)]
  \item[]
  \item\label{thm:triangle_free_i} For any $\eta > 4/9$ and $k = \lceil 2\eta + (1-\eta)(\Delta +1)\rceil$, there exist triangle-free graphs that are not $k$-recolourable.
  \item\label{thm:triangle_free_ii} For any $\eta \le 4/9$ and $k = \lceil 2\eta + (1-\eta)(\Delta +1)\rceil$, if a triangle-free graph admits a frozen $k$-colouring, then it is unique up to permuting colours.
  \end{enumerate}
\end{restatable}

Since \cref{thm:no_frozen} and \ref{thm:triangle_free_ii} guarantee that the main obstructions to recolourability do not occur in these ranges, they strongly motivate the following conjectures, which are tight because of the construction in \cite{MR4660623} and \ref{thm:triangle_free_i} respectively.
\begin{conjecture}\label{conj:recol_reed}
  Any graph $G$ is $k$-recolourable for all
  $k \ge \lceil \frac{1}{3}\omega(G) + {\frac{2}{3}(\Delta(G) +1)} \rceil$.
\end{conjecture}

\begin{conjecture}\label{conj:recol_triangle_free}
  Any triangle-free graph $G$ is $k$-recolourable for all
  $k \ge \lceil \frac{4}{9}\cdot 2 + {\frac{5}{9}(\Delta(G) +1)} \rceil$.
\end{conjecture}

As a consequence of \cref{thm:no_frozen} (respectively
\cref{thm:triangle_free}), note that disproving
\cref{conj:recol_reed} (respectively \cref{conj:recol_triangle_free}) would
require new tools and would significantly improve our understanding of Kempe
recolouring and its obstructions.  Finally, all advances towards Reed's
conjecture rely on the probabilistic method. This proof method does not adapt
well to reconfiguration proofs. For this reason, proving \cref{conj:recol_reed} or 
\cref{conj:recol_triangle_free} even for small positive $\eps$ seems challenging and is likely to either be the
first application of the probabilistic method to reconfiguration, or to yield a
constructive proof of Reed's colouring question, which would be of independent
interest.

\subsection{Recolouring odd-hole-free graphs}
A \emph{hole} is an induced cycle of length at least four. The complement of a hole is an \emph{antihole}, and an odd-(anti)hole is an (anti)hole of odd size. The Strong Perfect Graph Theorem~\cite{chudnovsky2006strong} characterises perfect graphs, that is graphs with chromatic number equal to their clique number, as graphs that have neither odd-holes nor odd-antiholes.

Aravind, Karthick and Subramanian~\cite{aravind2011Bounding} proved the Local Reed Conjecture in the class of odd-hole-free graphs. Weil~\cite{2016arXiv161102063W} then observed this result generalises immediately to the class $\mathcal{H}$ of graphs whose odd-holes all contain some vertex of degree less than $f(G)$. Bonamy, Kaiser and Legrand-Duchesne~\cite{legrand-duchesne2024Exploring} gave an alternative proof of this result using Kempe changes:
\begin{theorem}[Bonamy, Kaiser and Legrand-Duchesne~\cite{legrand-duchesne2024Exploring}]\label{thm:reed_odd_hole}
    Every $k$-colouring of any graph $G \in \mathcal{H}$ is Kempe equivalent to a $f(G)$-colouring. 
\end{theorem}
This result also implies that the graphs in $\mathcal{H}$, and in particular odd-hole-free graphs, have no frozen $k$-colourings for
$k > f(G)$. Again, as frozen colourings are the main known obstruction to recolourability, this motivated the following conjecture:
\begin{conjecture}[Bonamy, Kaiser and Legrand-Duchesne~\cite{legrand-duchesne2024Exploring}]\label{conj:odd_hole_reed}
  All odd-hole-free graphs of maximum degree $\Delta$ and clique number $\omega$ are
  $k$-recolourable for $k \ge \left\lceil \frac{\omega + \Delta + 1}{2}\right\rceil$.
\end{conjecture}
In~\cite{legrand-duchesne2024Exploring}, the same authors proved
a weakening of \cref{conj:odd_hole_reed}: all odd-hole-free graphs $G$ are~$k$-recolourable for~$k > \lceil (\chi(G) + \Delta(G) + 1)/2\rceil$, which settles  the special case of perfect graphs. Moreover, as Reed's conjecture holds for odd-hole-free graphs, this also implies that they are $k$-recolourable for~${k > \lceil(\omega(G) + 3(\Delta(G) +1))/4\rceil}$.
The main contribution of this paper improves this result and confirms \cref{conj:odd_hole_reed} up to using one extra colour (which is almost tight due to a folklore contruction of frozen colouring): 
\begin{restatable}{theorem}{maintheorem}
  \label{thm:main}
  \begin{enumerate}[label =(\roman*), ref=\cref{thm:main}.(\roman*)]
  \item[]
  \item\label{thm:main_i} Let $G$ be a graph in which every odd-hole contains a vertex of degree at most $f(G)$. Then~$G$ is~$k$-recolourable for all $k > f(G)$.
  \item\label{thm:main_ii} For all $k \ge 3$, there is a non-$k$-recolourable perfect graph $G$ with $k=\left\lceil \frac{\omega(G)+\Delta(G)+1}{2}\right\rceil-1$.
  \end{enumerate}
\end{restatable}
Therefore, the only remaining open case for \cref{conj:odd_hole_reed} is $k = f(G)$. Moreover, this also shows that \cref{thm:reed_odd_hole} is tight up to one colour, even for perfect graphs.

\subsection{Summary}
Given a class $\mathcal{G}$ of connected graphs, let $\eps^\star(\mathcal{G})$ be
the supremum of all $\eps$ such that there exists $c \ge 0$ such that all graphs in
$\mathcal{G}$ are $k$-recolourable for all
$k \ge \lceil \eps\omega + (1-\eps)(\Delta+1)\rceil + c$. Denote
$\eta^\star(\mathcal{G})$ be the infimum $\eta$ such that there exists a non-$k$-recolourable graph in $\mathcal{G}$ with a frozen $k$-colouring and
$k = \lceil \eta\omega + (1-\eta)(\Delta+1)\rceil$. We have
$\eps^\star(\mathcal{G}) \le \eta^\star(\mathcal{G})$ for any graph class
$\mathcal{G}$. Our results can be summarised as follows:

\begin{center}
\begin{tabular}{|l|C{4.5cm}|C{4.5cm}|C{4.5cm}|}
  \hline
  $\mathcal{G}$ & all graphs & triangle-free & odd-hole-free\\
  \hline
  $\eps^\star$ & \cref{conj:recol_reed} & \cref{conj:recol_triangle_free} & 1/2\quad (\ref{thm:main_i}) \\
  \hline
  $\eta^\star$ & 1/3\quad (\cref{thm:no_frozen})  & 4/9\quad (\cref{thm:triangle_free}) & 1/2\quad (\ref{thm:main_ii}) \\
  \hline
\end{tabular}
\end{center}

After providing some notations and basic definitions in \Cref{sec:prel}, we
prove \Cref{thm:main} in \Cref{sec:proof} and the bounds on $\eta^\star$  for general graphs and triangle free-graphs in \Cref{sec:tight,sec:triangle_free} respectively.

\section{Preliminaries}\label{sec:prel}

We now define the notation that we will be using in our proofs. Let~$G$ be a
graph.
If $v\in V(G)$, let $G\setminus v$ be the subgraph of $G$ obtained by removing $v$ as well as all edges incident to it.
Similarly, if $S\subset V(G)$, $G\setminus S$ is the subgraph of $G$ obtained by removing all vertices in $S$ as well as edges incident to them. 
We denote $N(u)$ and $N[u]$ the open and closed neighbourhoods of a vertex $u$
Given a vertex $v$ and subset $U$ of the vertices, we
denote $\deg(v,U) = |N(v) \cap U|$. Likewise, given two subset $U_1$ and $U_2$
of vertices, we denote $\deg(U_1,U_2)$ the number of edges between $U_1$ and $U_2$, and denote $\deg(U) = \deg(U,V(G)\setminus U)$.

A~\emph{$k$-colouring} of~$G$ is a
function~$\gamma : V(G) \rightarrow [k]$ such $\gamma(u) \neq \gamma(v)$ for all
$uv \in E(G)$. 
For any subgraph~$H$ of~$G$, we
let~$\gamma(H) \coloneqq \{\gamma(v): v \in V(H)\}$ be the set of colours used
in the subgraph~$H$, and let~$\gamma|_{H}$ be the colouring~$\gamma$ restricted
to~$H$.  We say that two colourings $\alpha$ and $\beta$ \emph{agree} on a set
$X \subseteq V(G)$ if $\alpha|_{X} = \beta|_{X}$. Similarly, we say that
$\alpha$ and $\beta$ \emph{differ} on $X$, if $X$ is the set of all vertices
$v \in V(G)$ such that $\alpha(v) \neq \beta(v)$.  We say that $u$
\emph{misses} the colour $c$ in the colouring $\gamma$ if
$c \notin \gamma(N[u])$.

An \emph{$\{a,b\}$-Kempe chain} $K$ of $\gamma$ is a connected component of the subgraph induced by vertices of colours $a$ or $b$ (see~\cref{fig:kempe}). We say that $a$ and $b$ are the colours \emph{used} in $K$. We denote~$K_{v, c}(G, \gamma)$ the $\{\gamma(v),c\}$-Kempe chain that contains the vertex $v$. Let~$K(G, \gamma)$ be the set of all Kempe chains in~$G$ under the
colouring~$\gamma$, where we also allow the empty Kempe chain (the graph with no vertices)
to be part of this set. For an $\{a,b\}$-Kempe chain~$K \in K(G, \gamma)$,
let~$\perform{\gamma}{K}$ be the colouring obtained after interchanging the colours $a$ and $b$ on
all the vertices in~$K$. We also refer to this operation as
performing the \emph{Kempe change}~$K$ in~$G$ under~$\gamma$. 
Given an induced subgraph $G'$ of $G$, a colouring $\gamma$ of $G$ and an $\{a,b\}$-Kempe chain $K$ of $\gamma|_{G'}$, the \emph{extension} of $K$ to $\gamma$ is the $\{a,b\}$-Kempe chain of $\gamma$ which contains the vertices of $K$.

We now state a lemma that will be used in \cref{sec:tight} and \cref{sec:triangle_free} to lower bound the
maximum degree of a graph: 
\begin{lemma}\label{lem:avgdeg_frozen}
 Any colour class $U$ of a frozen $k$-colouring $\alpha$ of an $n$-vertex graph $G$ has average degree 
 $$\frac{\deg(U)}{|U|} \ge \frac{n-|U| + (k-1)(|U|-1)}{|U|}.$$
\end{lemma}
\begin{proof}
For each other colour class $U'$, the graph $G[U\cup U']$ is connected because $\alpha$ is frozen. Hence $G[U\cup U']$ has at least $|U| + |U'| -1$ edges and 
$$\frac{\deg(U)}{|U|} \ge \sum_{U'} \frac{|U| + |U'|-1}{|U|}\ge\frac{n-|U| + (k-1)(|U|-1)}{|U|}.$$
\end{proof}

Finally, the following remark shows that the definition of $\eta^\star$ can be simplified:\begin{remark}\label{rem:etastar}
For any fixed graph $G$, the function ${\eta \mapsto \eta\omega(G) + (1-\eta)(\Delta(G) +1)}$
is non-increasing, so we may omit the ceiling in the definition of $\eta^\star$.
\end{remark}

\section{Recolourability of odd-hole-free graphs}\label{sec:proof}
This section is dedicated to the proof of \cref{thm:main}, that we recall:
\maintheorem*

Recall also that $f(G) = \max_{v \in V(G)} \left\lceil \left(\omega(v) + \deg(v) + 1\right) / 2\right\rceil$.
We first prove \ref{thm:main_i} and prove \ref{thm:main_ii} in \cref{ssec:odd_hole_tight}.

Let $G$ be a graph in which all odd-holes contain some vertex of degree at most $f(G)$ and let~$k$ be an integer satisfying~$k > f(G)$.
We proceed by induction on the
number of vertices of~$G$. 

If $G$ contains an odd-hole, let $v$ be a vertex of degree at most $f(G)$ in it. A classical result (namely Lemma 2.3 in \cite{VERGNAS198195}, used in a variety of other Kempe recolouring proofs) shows that $G$ is $k$-recolourable, as $k > f(G) \ge \deg(v)$ and $G\setminus v$ is $k$-recolourable by induction. Therefore one can assume that $G$ has no odd-holes.

Let~$\alpha$ and~$\beta$ be two~$k$-colourings of~$G$. We will transform~$\alpha$
into~$\beta$ using a sequence of Kempe chains in~$G$. 
Let~$v \in V(G)$ be a vertex of~$G$, and
let~$G' \coloneqq G \setminus v$. 
Let~$\alpha' \coloneqq \alpha|_{G'}$
and~$\beta' \coloneqq \beta|_{G'}$ be the restriction of colourings~$\alpha$
and~$\beta$ to~$G'$, respectively.
By the induction hypothesis, there exists a
sequence~$\alpha' = \alpha'_0, \alpha'_1, \dots, \alpha'_h = \beta'$ of $k$-colourings
of~$G'$ such that~$\alpha'_i = \perform{\alpha'_{i - 1}}{K_i}$, for some sequence of Kempe
chains~$K_1, K_2, \dots, K_h$ with~$K_i \in K(G', \alpha'_{i - 1})$ for~$i = 1, \dots, h$.

A natural strategy to recolour $\alpha$ into $\beta$ would be to apply the Kempe changes $K_1, \dots, K_h$ to $\alpha$. Unfortunately, any of these Kempe chains might not be a valid Kempe chain in $G$ if it uses the colour of $v$. To circumvent this issue, the key idea is to allow for some controlled errors at each step. We formalise this error control with the following definition. We say that a colouring $\gamma$ of $G$ is \emph{faithful} to a colouring $\gamma'$ of $G'$ if there exist two colours $a$ and $b$ such that the following conditions hold:
\begin{enumerate}[label ={{\emph{\textbf{C\arabic{enumi}}}}}, ref=\textbf{C\arabic{enumi}}]   
  \item\label{item:one-away} The colourings $\gamma|_{G'}$ and $\gamma'$ differ on a set $\cB$ of $\{a, b\}$-Kempe chains under $\gamma'$, and the vertices of $\cB$ only use colours $a,b$ under $\gamma$. We call~$\cB$ the \emph{bad} Kempe chains for $\gamma$ and $\gamma'$.
  \item\label{item:bad-chain-is-bad} Every Kempe chain in $\cB$ contains a neighbour of $v$ coloured $a$ in $\gamma'$.
\end{enumerate}

We will construct a sequence of colourings~$\alpha = \alpha_0, \alpha_1, \dots, \alpha_h$ of~$G$ such that each $\alpha_i$ is faithful to $\alpha_i'$ and for any $1 \le i \le h$, the colourings $\alpha_i$ and $\alpha_{i-1}$ are Kempe equivalent in $G$.
Assume that we have already constructed the sequence $\alpha_0, \dots, \alpha_{i-1}$ for some $1 \le i \le h$. We use two lemmas to define $\alpha_i$: \cref{lem:base} in the favourable case where $\alpha_{i-1}'$ and $\alpha_{i-1}$ agree on $G'$, and \cref{lem:induction} in the more involved case where $\alpha_{i-1}$ is faithful to $\alpha_{i-1}'$ without extending it. These two lemmas will be proved in \cref{sec:restriction} and \cref{sec:faithful} respectively.

Once the sequence $\alpha = \alpha_0, \alpha_1, \dots, \alpha_h$ is constructed, we only need to argue that $\beta$ is Kempe equivalent to $\alpha_h$, in order to conclude the proof of \ref{thm:main_i}.
We first argue that $\alpha_h$ is Kempe equivalent to a colouring $\gamma$ that agrees with $\beta$ on $G'$, then $\beta$ can be obtained from $\gamma$ by simply recolouring $v$ to its final colour.
Let $\cB$ be the set of bad $\{a,b\}$-Kempe chains for $\alpha_h$ and $\alpha_h'$. If these Kempe chains are also Kempe chains of $G$, then performing them results in a colouring $\gamma$ that agrees with $\beta$ on $G'$. If not, then $v$ is coloured $a$ or $b$ and as $\beta$ is a colouring of $G$, this implies that $\cB$ contains all $\{a,b\}$-Kempe chains of $\alpha_h'$ that contain a neighbour of $v$. Therefore there exists an $\{a,b\}$-Kempe chain $K$ of $\alpha_h$ whose restriction to $G'$ is $\cB$. As a result, $\gamma = \perform{\alpha_h}{K}$ agrees with $\beta$ on $G'$, which concludes the proof of \ref{thm:main_i}.

\subsection{Defining \texorpdfstring{$\alpha_i$}{alpha(i)} when \texorpdfstring{$\alpha_{i-1}$}{alpha(i-1)} and \texorpdfstring{$\alpha'_{i-1}$}{alpha'(i-1)} agree on \texorpdfstring{$G'$}{G'}}\label{sec:restriction}
\begin{lemma}\label{lem:base}
Let $G$ be an odd-hole-free graph, $v\in V(G)$ arbitrary and $G'\coloneqq G\setminus v$.
  Let $\gamma$ and $\gamma'$ be colourings of $G$ and $G'$ respectively, such that $\gamma|_{G'} = \gamma'$. Let $K'$ be an $\{a,b\}$-Kempe chain of $\gamma'$, and let~$K$ be the extension of~$K'$ in~$\gamma$. Then the colouring $\delta \coloneq \perform{\gamma}{K}$ is faithful to $\delta' \coloneq \perform{\gamma'}{K'}$.
\end{lemma}
\begin{proof}
  If $K'$ is also a Kempe chain in $G$ under $\gamma$, i.e.\ $K=K'$, then $\delta = \perform{\gamma}{K'}$ and its restriction to $G'$ is $\delta'$.
  In particular, $\delta$ is faithful to $\delta'$. Hence, we can assume that $K$ contains $v$. 
  Up to relabelling colours, assume that $\gamma(v) = 0$ and the other colour used by $K$ is $1$. 
  Let $\cA$ be the set of connected components of $K \setminus v$ (see \cref{fig:basea}). Note that $\cA$ is the set of all $\{0,1\}$-Kempe chains of $\gamma'$ that contain a neighbour of $v$ coloured 1.
  By performing $K$ in $G$ under $\gamma$, we obtain a colouring $\delta$ whose restriction to $G'$ differs with $\delta'$ on the chains of $\cB \coloneq \cA \setminus \{K'\}$ (see \cref{fig:baseb,fig:basec}), so \ref{item:one-away} holds. By construction, every Kempe chain in $\cB$ contains a neighbour of $v$ coloured $1$ in $\delta'$, so \ref{item:bad-chain-is-bad} also holds.
\end{proof}

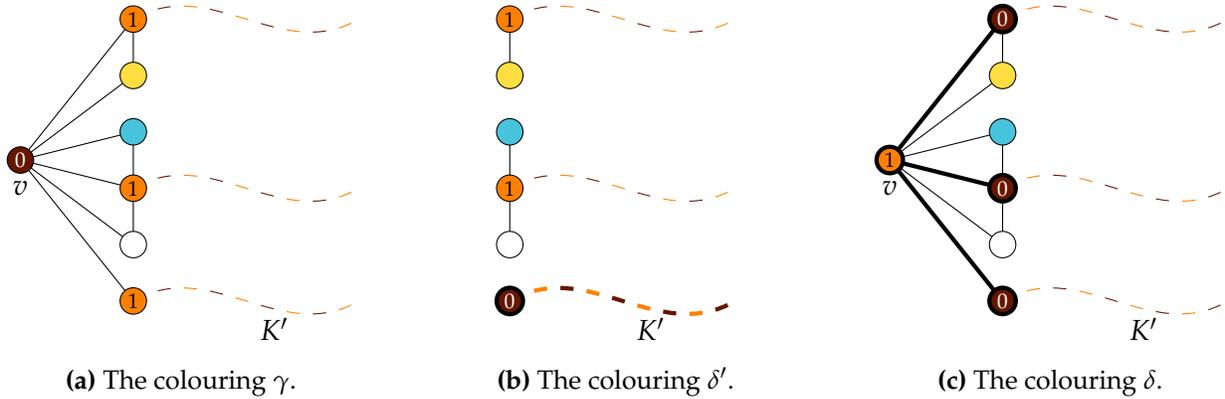
\begin{figure}[h!]
    \begin{subfigure}[t]{.3\textwidth}
        \centering
        \begin{tikzpicture}[scale=1.5]
        \coordinate (v) at (0,0);
        \foreach \x in {1,...,6}{
            \coordinate (a\x) at (1,-1.75+.5*\x);
        }

        \draw (a2) -- (a3) -- (a4) (a5) -- (a6);
        \foreach \x in {1,...,6}{
            \draw (v) -- (a\x);}

        \node[below=4 pt of v] {$v$};       

        \draw[dashone, \colord, postaction={draw=\colorc, dash phase= \dash}] 
        (a1) .. controls ++(30:.8cm) and ++(210:.8cm) .. (3,-1.25);
        \draw[dashone, \colord, postaction={draw=\colorc, dash phase= \dash}] 
        (a3) .. controls ++(30:.8cm) and ++(210:.8cm) .. (3,-.25);
        \draw[dashone, \colord, postaction={draw=\colorc, dash phase= \dash}] 
        (a6) .. controls ++(30:.8cm) and ++(210:.8cm) .. (3,1.25);
        \node at (2.25,-1.5) {$K'$};
  
        \node[circle,draw=black,inner sep = 1pt, minimum size = 2*\ray,fill=\colorc]  at (v) {\scriptsize \textcolor{white} 0};
        \node[circle,draw=black,inner sep = 1pt, minimum size = 2*\ray,fill=\colord]  at (a1) {\scriptsize 1};
        \node[circle,draw=black,inner sep = 1pt, minimum size = 2*\ray,fill=\colorf]  at (a2) {};
        \node[circle,draw=black,inner sep = 1pt, minimum size = 2*\ray,fill=\colord]  at (a3) {\scriptsize 1};
        \node[circle,draw=black,inner sep = 1pt, minimum size = 2*\ray,fill=\colorb]  at (a4) {};
        \node[circle,draw=black,inner sep = 1pt, minimum size = 2*\ray,fill=\colora]  at (a5) {};
        \node[circle,draw=black,inner sep = 1pt, minimum size = 2*\ray,fill=\colord]  at (a6) {\scriptsize 1};
        \end{tikzpicture}
        \caption{The colouring $\gamma$.}\label{fig:basea}
    \end{subfigure}
    \hfill
   \begin{subfigure}[t]{.3\textwidth}
        \centering
        \begin{tikzpicture}[scale=1.5]
        \foreach \x in {1,...,6}{
            \coordinate (a\x) at (1,-1.75+.5*\x);
        }

        \draw (a2) -- (a3) -- (a4) (a5) -- (a6);

        \draw[ultra thick, dashone, \colorc, postaction={draw=\colord, dash phase= \dash}] 
        (a1) .. controls ++(30:.8cm) and ++(210:.8cm) .. (3,-1.25);
        \draw[dashone, \colord, postaction={draw=\colorc, dash phase= \dash}] 
        (a3) .. controls ++(30:.8cm) and ++(210:.8cm) .. (3,-.25);
        \draw[dashone, \colord, postaction={draw=\colorc, dash phase= \dash}] 
        (a6) .. controls ++(30:.8cm) and ++(210:.8cm) .. (3,1.25);
        \node at (2.25,-1.5) {$K'$};
  
        \node[ultra thick,circle,draw=black,inner sep = 1pt, minimum size = 2*\ray,fill=\colorc]  at (a1) {\scriptsize \textcolor{white} 0};
        \node[circle,draw=black,inner sep = 1pt, minimum size = 2*\ray,fill=\colorf]  at (a2) {};
        \node[circle,draw=black,inner sep = 1pt, minimum size = 2*\ray,fill=\colord]  at (a3) {\scriptsize 1};
        \node[circle,draw=black,inner sep = 1pt, minimum size = 2*\ray,fill=\colorb]  at (a4) {};
        \node[circle,draw=black,inner sep = 1pt, minimum size = 2*\ray,fill=\colora]  at (a5) {};
        \node[circle,draw=black,inner sep = 1pt, minimum size = 2*\ray,fill=\colord]  at (a6) {\scriptsize 1};
        \end{tikzpicture}
        \caption{The colouring $\delta'$.}\label{fig:baseb}
    \end{subfigure}   
    \hfill
    \begin{subfigure}[t]{.3\textwidth}
        \centering        
        \begin{tikzpicture}[scale=1.5]
        \coordinate (v) at (0,0);
        \foreach \x in {1,...,6}{
            \coordinate (a\x) at (1,-1.75+.5*\x);
        }

        \draw (a2) -- (a3) -- (a4) (a5) -- (a6);
        \foreach \x in {1,3,6}{
            \draw[ultra thick] (v) -- (a\x);}
        \foreach \x in {2,4,5}{
            \draw (v) -- (a\x);}

        \node[below=4 pt of v] {$v$};       

        \draw[dashone, \colorc, postaction={draw=\colord, dash phase= \dash}] 
        (a1) .. controls ++(30:.8cm) and ++(210:.8cm) .. (3,-1.25);
        \draw[dashone, \colorc, postaction={draw=\colord, dash phase= \dash}] 
        (a3) .. controls ++(30:.8cm) and ++(210:.8cm) .. (3,-.25);
        \draw[dashone, \colorc, postaction={draw=\colord, dash phase= \dash}] 
        (a6) .. controls ++(30:.8cm) and ++(210:.8cm) .. (3,1.25);
        \node at (2.25,-1.5) {$K'$};
  
        \node[ultra thick, circle,draw=black,inner sep = 1pt, minimum size = 2*\ray,fill=\colord]  at (v) {\scriptsize 1};
        \node[ultra thick, circle,draw=black,inner sep = 1pt, minimum size = 2*\ray,fill=\colorc]  at (a1) {\scriptsize \textcolor{white} 0};
        \node[circle,draw=black,inner sep = 1pt, minimum size = 2*\ray,fill=\colorf]  at (a2) {};
        \node[ultra thick, circle,draw=black,inner sep = 1pt, minimum size = 2*\ray,fill=\colorc]  at (a3) {\scriptsize \textcolor{white} 0};
        \node[circle,draw=black,inner sep = 1pt, minimum size = 2*\ray,fill=\colorb]  at (a4) {};
        \node[circle,draw=black,inner sep = 1pt, minimum size = 2*\ray,fill=\colora]  at (a5) {};
        \node[ultra thick, circle,draw=black,inner sep = 1pt, minimum size = 2*\ray,fill=\colorc]  at (a6) {\scriptsize \textcolor{white} 0};
        \end{tikzpicture}
        \caption{The colouring $\delta$.}\label{fig:basec}
    \end{subfigure}
    \caption{The different colourings of \cref{lem:base}. The set of Kempe chains in $\cA$ are represented by dashed lines. The Kempe chains on which $\delta$ and $\delta'$ differ with $\gamma$ are thickened.}\label{fig:base}
\end{figure}

In particular, by defining~$\alpha_0 = \alpha$,~\cref{lem:base} builds a colouring $\alpha_1$ Kempe equivalent to $\alpha_0$ and faithful to $\alpha_1'$.

\subsection{Defining \texorpdfstring{$\alpha_i$}{alpha(i)} when \texorpdfstring{$\alpha_{i-1}$}{alpha(i-1)} and \texorpdfstring{$\alpha'_{i-1}$}{alpha' (i-1)} differ on $G'$}\label{sec:faithful}

In the case where $\alpha_{i-1}$ is faithful to $\alpha_{i-1}'$ but they differ on $G'$, the main difficulty arises when the Kempe change $K'$ we need to perform would affect vertices of $\cB$.
Under this setting, we first
perform a series of Kempe changes to obtain a colouring which agrees with $\alpha_{i-1}'$ on $\cB$, possibly creating a different set of bad Kempe chains in the process,
before performing $K'$ to obtain a colouring $\alpha_i$ faithful to $\alpha_i'$. This procedure is handled by the following lemma:



\begin{lemma}\label{lem:induction}
    Let $G$ be an odd-hole free graph, $G'\coloneqq G\setminus v$ for some $v\in V(G)$ and $k>f(G)$.
    Let $\gamma_1'$, $\gamma_2'$ be two colourings of $G'$ that differ by one Kempe change $K'$ and $\gamma_1$ be a colouring of $G$ faithful to $\gamma_1'$. Then there is a colouring $\gamma_2$ of $G$ that is faithful to $\gamma_2'$ and Kempe equivalent to $\gamma_1$.
\end{lemma}
\begin{proof}
Up to relabelling the colours, we can assume that the set $\cB$ of bad Kempe chains under $\gamma_1'$ are using the colours $0$ and $1$ and that every Kempe chain in $\cB$ contains a neighbour of $v$ coloured $1$ in $\gamma_1'$.

\begin{claim}\label{cl:case2}
    If $\gamma_1(v) \neq 0$ then the conclusion of \cref{lem:induction} holds.
\end{claim}
\begin{poc}
    Assume that $\gamma_1(v) \neq 0$. Then the bad Kempe chains in $\cB$ are also Kempe chains of $G$ under $\gamma_1$. So we can successively perform all the Kempe chains in $\cB$ under $\gamma_1$ to obtain a colouring $\delta$ such that $\delta|_{G'} = \gamma_1'$.
    By~\cref{lem:base}, $\gamma_2 = \perform{\delta}{K'}$ is faithful to $\gamma_2'$ and Kempe equivalent to $\gamma_1$.
\end{poc}
\begin{claim}\label{cl:case1}
    If $K'$ does not use the colours of the bad Kempe chains, then the conclusion of \cref{lem:induction} holds.
\end{claim}
\begin{poc}
    Assume that $K'$ uses neither $0$ nor $1$.
    By \cref{cl:case2}, we can assume that $\gamma_1(v) = 0$. 
    Then, $K'$ is also a Kempe chain under $\gamma_1$ and it does not intersect with the Kempe chains in $\cB$.
    So $\gamma_2 \coloneq \perform{\gamma_1}{K'}$ satisfies the conditions \ref{item:one-away} and \ref{item:bad-chain-is-bad}, by taking identical $a$, $b$ and $\cB$.
\end{poc}

Call the colours of $\{0,1\} \cup \gamma'_1(K')$ the \emph{problematic colours}.
By \cref{cl:case2} and \cref{cl:case1} we may assume from now on that $\gamma_1(v)=0$ and that there are at most $3$ problematic colours.

If $v$ misses colour $c$ in $\gamma_1$, that is $c \notin  \gamma_1(N[v])$, then we can recolour $v$ into $c$ without changing $\gamma_1|_{G'}$ and conclude using \cref{cl:case2}. Thus, we can assume that the  neighbourhood of $v$ contains all colours except $0$ in $\gamma_1$,
i.e. ~$|\gamma_1(N(v))| = k - 1$.

Likewise, if $v$ has only one neighbour $w$ coloured $1$ in $\gamma_1$ then 
 by \ref{item:one-away}, $\gamma_1|_{G'}$ and $\gamma_1'$ differ on a unique bad Kempe chain $B' \in \cB$ (see \cref{fig:case3aeasy}). Let $B$ be the extension of $B'$ to $\gamma_1$ and $\delta = \perform{\gamma_1}{B}$. Then $\delta|_{G'} = \gamma_1'$ (see \cref{fig:case3beasy}) and $\gamma_2$ is given by \cref{lem:base}. Thus we can also assume that $v$ has at least two neighbours coloured $1$.
\begin{figure}[h!]
    \begin{subfigure}[t]{.45\textwidth}
    \centering
    \begin{tikzpicture}[scale=1.5]
        \coordinate (v) at (0,0);
        \foreach \x in {1,...,6}{
            \coordinate (a\x) at (1,-1.75+.5*\x);
        }

        \draw (a2) -- (a3) -- (a4) (a5) -- (a6);
        \foreach \x in {1,...,6}{
            \draw (v) -- (a\x);}

        \node[below=4 pt of v] {$v$};

        \draw[dashone, \colord, postaction={draw=\colorc, dash phase= \dash}] 
        (a3) .. controls ++(30:.8cm) and ++(210:.8cm) .. (3,-.25);
        \draw[dashone, \colorc, postaction={draw=\colorb, dash phase= \dash}] 
        (2,.5) .. controls ++(-60:.8cm) and ++(120:.8cm) .. (2,-1.5);
        \node at (2.25,-1.25) {$K'$};
  
        \node[circle,draw=black,inner sep = 1pt, minimum size = 2*\ray,fill=\colorc]  at (v) {\scriptsize \textcolor{white} 0};
        \node[circle,draw=black,inner sep = 1pt, minimum size = 2*\ray,fill=\colorb]  at (a1) {};
        \node[circle,draw=black,inner sep = 1pt, minimum size = 2*\ray,fill=\colorf]  at (a2) {};
        \node[circle,draw=black,inner sep = 1pt, minimum size = 2*\ray,fill=\colord]  at (a3) {\scriptsize 1};
        \node[circle,draw=black,inner sep = 1pt, minimum size = 2*\ray,fill=\colore]  at (a4) {};
        \node[circle,draw=black,inner sep = 1pt, minimum size = 2*\ray,fill=\colora]  at (a5) {};
        \node[circle,draw=black,inner sep = 1pt, minimum size = 2*\ray,fill=\colorb]  at (a6) {};
  
    \node[circle,draw=black,inner sep = 1pt, minimum size = \ray,fill=\colord]  at (2.07,-.27) {};
    
    \end{tikzpicture}
    \caption{The colouring $\gamma_1$. }\label{fig:case3aeasy}
    \end{subfigure}
    \hfill
    \begin{subfigure}[t]{.45\textwidth}
        \centering
        \begin{tikzpicture}[scale=1.5]
            \coordinate (v) at (0,0);
            \foreach \x in {1,...,6}{
                \coordinate (a\x) at (1,-1.75+.5*\x);}

            \draw (a2) -- (a3) -- (a4) (a5) -- (a6);
            \foreach \x in {1,...,6}{
                \draw (v) -- (a\x);}

            \node[below=4 pt of v] {$v$};
            
            \draw[dashone, \colorc, postaction={draw=\colord, dash phase= \dash}] 
            (a3) .. controls ++(30:.8cm) and ++(210:.8cm) .. (3,-.25);
            \draw[dashone, \colorc, postaction={draw=\colorb, dash phase= \dash}] 
            (2,.5) .. controls ++(-60:.8cm) and ++(120:.8cm) .. (2,-1.5);
            \node at (2.25,-1.25) {$K'$};
  
            \node[circle,draw=black,inner sep = 1pt, minimum size = 2*\ray,fill=\colord]  at (v) {\scriptsize 1};
            \node[circle,draw=black,inner sep = 1pt, minimum size = 2*\ray,fill=\colorb]  at (a1) {};
            \node[circle,draw=black,inner sep = 1pt, minimum size = 2*\ray,fill=\colorf]  at (a2) {};
            \node[circle,draw=black,inner sep = 1pt, minimum size = 2*\ray,fill=\colorc]  at (a3) {\scriptsize \textcolor{white} 0};
            \node[circle,draw=black,inner sep = 1pt, minimum size = 2*\ray,fill=\colore]  at (a4) {};
            \node[circle,draw=black,inner sep = 1pt, minimum size = 2*\ray,fill=\colora]  at (a5) {};
            \node[circle,draw=black,inner sep = 1pt, minimum size = 2*\ray,fill=\colorb]  at (a6) {};

    \node[circle,draw=black,inner sep = 1pt, minimum size = \ray,fill=\colorc]  at (2.07,-.27) {};

        \end{tikzpicture}
        \caption{The colouring $\delta$ extends $\gamma_1'$.}\label{fig:case3beasy}
    \end{subfigure}   
    \caption{Recolouring sequence when $v$ has only one neighbour coloured $1$. The smaller vertex represents any vertex in the intersection of $B'$ and $K'$.}\label{fig:case3easy}
\end{figure}
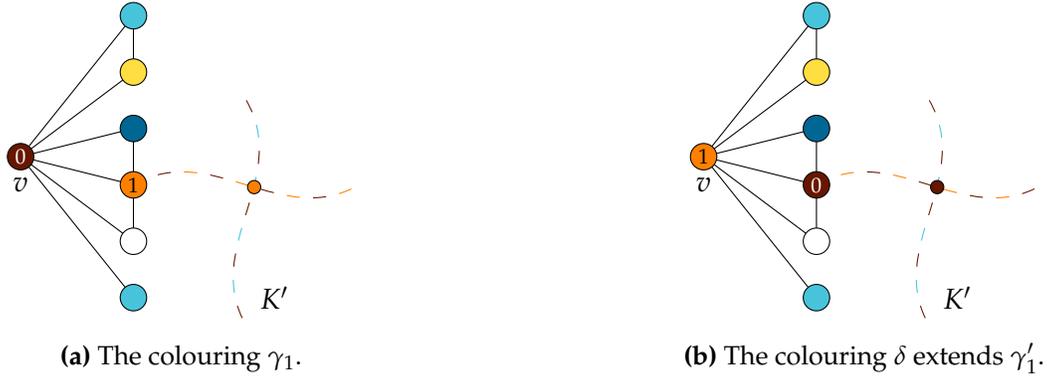

In order to recolour the vertex~$v$ and reduce to \cref{cl:case2}, we find a Kempe chain~$L$ of $\gamma_1$ such that $v$ misses a problematic colour $c$ in $\perform{\gamma_1}{L}$. We then recolour~$v$ with this $c$, perform the Kempe changes in $\cB$ and~$K'$, and maintain~$L$ as the
next bad Kempe chain (i.e., define~$\cB$ as~$\{L\}$).
We now formalise this. 
Let $S$ be the set of neighbours of $v$ whose colour appears only once in $\gamma_1(N(v))$. Recall that $v$ has at least two neighbours coloured $1$ in $\gamma_1$, so these vertices do not belong to $S$. We have
\begin{align*}
  \deg(v) = |N(v)| & \geq 2 (k - 1 - |S|) + |S| \\
         &\ge 2 \left\lceil\frac{\omega(v)+\deg(v) + 1}{2}\right\rceil - |S| \\
         &> \omega(v) + \deg(v) - |S|,
\end{align*}
which implies that~$|S| \geq \omega(v) + 1$.
Let $T$ be the vertices of $S$ whose colour under $\gamma_1$ avoids the problematic colours.
Since there are at most three problematic colours and $S$ contains no vertices coloured $1$ or $0$, 
we have $|T| \ge \omega(v)$.
The vertices of $T$ induce a non-edge $uw$, as they would otherwise induce a clique of size at least $\omega(v)+1$ together with $v$.
Say $u$ and $w$ are respectively coloured $2$ and $3$ in $\gamma_1$ (see~\cref{fig:case3a}).

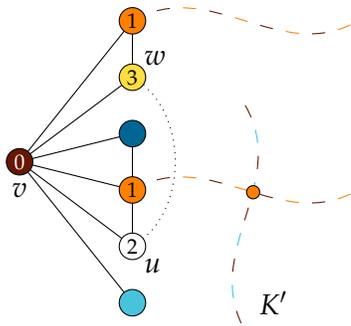
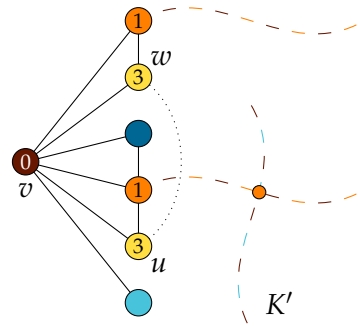
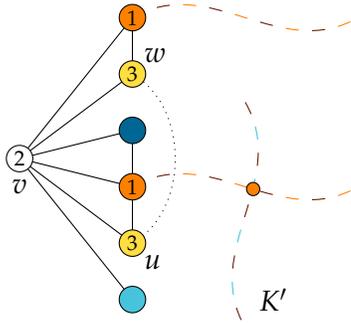
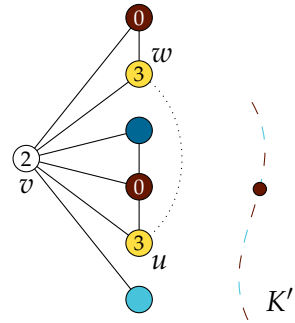
\begin{figure}[h!]
    \begin{subfigure}[t]{.45\textwidth}
    \centering
    \begin{tikzpicture}[scale=1.5]
        \coordinate (v) at (0,0);
        \foreach \x in {1,...,6}{
            \coordinate (a\x) at (1,-1.75+.5*\x);
        }

        \draw (a2) -- (a3) -- (a4) (a5) -- (a6);
        \foreach \x in {1,...,6}{
            \draw (v) -- (a\x);}

        \node[below=4 pt of v] {$v$};
        \node[below right=1 pt of a2] {$u$};
        \node[above right=1 pt of a5] {$w$};
        \draw (a2) edge[dotted, bend right=60] (a5);

        \draw[dashone, \colord, postaction={draw=\colorc, dash phase= \dash}] 
        (a3) .. controls ++(30:.8cm) and ++(210:.8cm) .. (3,-.25);
        \draw[dashone, \colord, postaction={draw=\colorc, dash phase= \dash}] 
        (a6) .. controls ++(30:.8cm) and ++(210:.8cm) .. (3,1.25);
        \draw[dashone, \colorc, postaction={draw=\colorb, dash phase= \dash}] 
        (2,.5) .. controls ++(-60:.8cm) and ++(120:.8cm) .. (2,-1.5);
        \node at (2.25,-1.25) {$K'$};
  
        \node[circle,draw=black,inner sep = 1pt, minimum size = 2*\ray,fill=\colorc]  at (v) {\scriptsize \textcolor{white} 0};
        \node[circle,draw=black,inner sep = 1pt, minimum size = 2*\ray,fill=\colorb]  at (a1) {};
        \node[circle,draw=black,inner sep = 1pt, minimum size = 2*\ray,fill=\colorf]  at (a2) {\scriptsize 2};
        \node[circle,draw=black,inner sep = 1pt, minimum size = 2*\ray,fill=\colord]  at (a3) {\scriptsize 1};
        \node[circle,draw=black,inner sep = 1pt, minimum size = 2*\ray,fill=\colore]  at (a4) {};
        \node[circle,draw=black,inner sep = 1pt, minimum size = 2*\ray,fill=\colora]  at (a5) {\scriptsize 3};
        \node[circle,draw=black,inner sep = 1pt, minimum size = 2*\ray,fill=\colord]  at (a6) {\scriptsize 1};

        \node[circle,draw=black,inner sep = 1pt, minimum size = \ray,fill=\colord]  at (2.07,-.27) {};

    \end{tikzpicture}
    \caption{The colouring $\gamma_1$. Vertices $u$ and $w$ are non adjacent.}\label{fig:case3a}
    \end{subfigure}
    \hfill
    \begin{subfigure}[t]{.45\textwidth}
        \centering
        \begin{tikzpicture}[scale=1.5]
            \coordinate (v) at (0,0);
            \foreach \x in {1,...,6}{
                \coordinate (a\x) at (1,-1.75+.5*\x);}

            \draw (a2) -- (a3) -- (a4) (a5) -- (a6);
            \foreach \x in {1,...,6}{
                \draw (v) -- (a\x);}

            \node[below=4 pt of v] {$v$};
            \node[below right=1 pt of a2] {$u$};
            \node[above right=1 pt of a5] {$w$};
            \draw (a2) edge[dotted, bend right=60] (a5);

            \draw[dashone, \colord, postaction={draw=\colorc, dash phase= \dash}] 
            (a3) .. controls ++(30:.8cm) and ++(210:.8cm) .. (3,-.25);
            \draw[dashone, \colord, postaction={draw=\colorc, dash phase= \dash}] 
            (a6) .. controls ++(30:.8cm) and ++(210:.8cm) .. (3,1.25);
            \draw[dashone, \colorc, postaction={draw=\colorb, dash phase= \dash}] 
            (2,.5) .. controls ++(-60:.8cm) and ++(120:.8cm) .. (2,-1.5);
            \node at (2.25,-1.25) {$K'$};
  
            \node[circle,draw=black,inner sep = 1pt, minimum size = 2*\ray,fill=\colorc]  at (v) {\scriptsize \textcolor{white} 0};
            \node[circle,draw=black,inner sep = 1pt, minimum size = 2*\ray,fill=\colorb]  at (a1) {};
            \node[circle,draw=black,inner sep = 1pt, minimum size = 2*\ray,fill=\colora]  at (a2) {\scriptsize 3};
            \node[circle,draw=black,inner sep = 1pt, minimum size = 2*\ray,fill=\colord]  at (a3) {\scriptsize 1};
            \node[circle,draw=black,inner sep = 1pt, minimum size = 2*\ray,fill=\colore]  at (a4) {};
            \node[circle,draw=black,inner sep = 1pt, minimum size = 2*\ray,fill=\colora]  at (a5) {\scriptsize 3};
            \node[circle,draw=black,inner sep = 1pt, minimum size = 2*\ray,fill=\colord]  at (a6) {\scriptsize 1};

            \node[circle,draw=black,inner sep = 1pt, minimum size = \ray,fill=\colord]  at (2.07,-.27) {};

        \end{tikzpicture}
        \caption{The colouring $\delta$ obtained by performing the $\{2,3\}$-Kempe chain containing $u$. Vertex $v$ now misses colour 2.}\label{fig:case3b}
    \end{subfigure}   
    
    \begin{subfigure}[t]{.45\textwidth}
        \centering        
        \begin{tikzpicture}[scale=1.5]
            \coordinate (v) at (0,0);
            \foreach \x in {1,...,6}{
                \coordinate (a\x) at (1,-1.75+.5*\x);}

            \draw (a2) -- (a3) -- (a4) (a5) -- (a6);
            \foreach \x in {1,...,6}{
            \draw (v) -- (a\x);}

            \node[below=4 pt of v] {$v$};
            \node[below right=1 pt of a2] {$u$};
            \node[above right=1 pt of a5] {$w$};
            \draw (a2) edge[dotted, bend right=60] (a5);

            \draw[dashone, \colord, postaction={draw=\colorc, dash phase= \dash}] 
            (a3) .. controls ++(30:.8cm) and ++(210:.8cm) .. (3,-.25);
            \draw[dashone, \colord, postaction={draw=\colorc, dash phase= \dash}] 
            (a6) .. controls ++(30:.8cm) and ++(210:.8cm) .. (3,1.25);
            \draw[dashone, \colorc, postaction={draw=\colorb, dash phase= \dash}] 
            (2,.5) .. controls ++(-60:.8cm) and ++(120:.8cm) .. (2,-1.5);
            \node at (2.25,-1.25) {$K'$};
  
            \node[circle,draw=black,inner sep = 1pt, minimum size = 2*\ray,fill=\colorf]  at (v) {\scriptsize 2};
            \node[circle,draw=black,inner sep = 1pt, minimum size = 2*\ray,fill=\colorb]  at (a1) {};
            \node[circle,draw=black,inner sep = 1pt, minimum size = 2*\ray,fill=\colora]  at (a2) {\scriptsize 3};
            \node[circle,draw=black,inner sep = 1pt, minimum size = 2*\ray,fill=\colord]  at (a3) {\scriptsize 1};
            \node[circle,draw=black,inner sep = 1pt, minimum size = 2*\ray,fill=\colore]  at (a4) {};
            \node[circle,draw=black,inner sep = 1pt, minimum size = 2*\ray,fill=\colora]  at (a5) {\scriptsize 3};
            \node[circle,draw=black,inner sep = 1pt, minimum size = 2*\ray,fill=\colord]  at (a6) {\scriptsize 1};
            
            \node[circle,draw=black,inner sep = 1pt, minimum size = \ray,fill=\colord]  at (2.07,-.27) {};
        \end{tikzpicture}
        \caption{The colouring $\tau$ obtained by recolouring $v$ with $2$.}\label{fig:case3c}
    \end{subfigure}
    \hfill
    \begin{subfigure}[t]{.45\textwidth}
        \centering        
        \begin{tikzpicture}[scale=1.5]
            \coordinate (v) at (0,0);
            \foreach \x in {1,...,6}{
            \coordinate (a\x) at (1,-1.75+.5*\x);}

            \draw (a2) -- (a3) -- (a4) (a5) -- (a6);
            \foreach \x in {1,...,6}{
            \draw (v) -- (a\x);}

            \node[below=4 pt of v] {$v$};
            \node[below right=1 pt of a2] {$u$};
            \node[above right=1 pt of a5] {$w$};
            \draw (a2) edge[dotted, bend right=60] (a5);

            \draw[dashone, \colorc, postaction={draw=\colorb, dash phase= \dash}] 
            (2,.5) .. controls ++(-60:.8cm) and ++(120:.8cm) .. (2,-1.5);
            \node at (2.25,-1.25) {$K'$};
  
            \node[circle,draw=black,inner sep = 1pt, minimum size = 2*\ray,fill=\colorf]  at (v) {\scriptsize 2};
            \node[circle,draw=black,inner sep = 1pt, minimum size = 2*\ray,fill=\colorb]  at (a1) {};
            \node[circle,draw=black,inner sep = 1pt, minimum size = 2*\ray,fill=\colora]  at (a2) {\scriptsize 3};
            \node[circle,draw=black,inner sep = 1pt, minimum size = 2*\ray,fill=\colorc]  at (a3) {\scriptsize \textcolor{white} 0};
            \node[circle,draw=black,inner sep = 1pt, minimum size = 2*\ray,fill=\colore]  at (a4) {};
            \node[circle,draw=black,inner sep = 1pt, minimum size = 2*\ray,fill=\colora]  at (a5) {\scriptsize 3};
            \node[circle,draw=black,inner sep = 1pt, minimum size = 2*\ray,fill=\colorc]  at (a6) {\scriptsize \textcolor{white} 0};
            \node[circle,draw=black,inner sep = 1pt, minimum size = \ray,fill=\colorc]  at (2.07,-.27) {};

        \end{tikzpicture}
        \caption{The colouring $\sigma$ obtained by recolouring the bad $\{0,1\}$-Kempe chains.}\label{fig:case3d}
    \end{subfigure}
    \caption{The recolouring sequence when $v$ has at least two neighours coloured $1$.}\label{fig:case3}
\end{figure}
Let $L$ be the $\{2,3\}$-Kempe chain in~$\gamma_1$ that contains $u$ and $\delta$ be the colouring $\perform{\gamma_1}{L}$, as in~\cref{fig:case3b}. Note that $L$ does not contain $w$, as otherwise $L \cup \{v\}$ would contain an induced odd cycle which would also be an induced odd cycle of G, a contradiction. As the colours $2$ and $3$ appear only once in the neighbourhood of $v$ in $\gamma_1$, the vertex $v$ misses the colour $2$ in $\delta$. Let $\sigma$ be the colouring obtained from $\delta$ by recolouring $v$ into $2$ (see \cref{fig:case3c}) and $\tau$ be the colouring obtained from $\sigma$ by performing the Kempe chains in $\cB$ (see \cref{fig:case3d}). The colouring $\sigma|_{G'}$ differs from $\gamma_1'$ only on the $\{2,3\}$-Kempe chain $L$, so $\sigma$ is faithful to $\gamma_1'$. Finally, by definition colours $2$ and $3$ do not appear in $K'$ under $\gamma_1'$, so~\cref{cl:case1} applied to $\sigma$ yields the desired $\gamma_2$.
\end{proof}

\subsection{Frozen colourings of odd-hole-free graphs with low degree and clique number}\label{ssec:odd_hole_tight}

\ref{thm:main_i} implies $\eps^\star(\{\text{odd-hole-free graphs}\}) \ge 1/2$. We now prove \ref{thm:main_ii}, which implies that $\eta^\star(\{\text{odd-hole-free graphs}\}) \le 1/2$ and thus that $\eps^\star = \eta^\star = 1/2$ for odd-hole-free graphs.

\begin{proof}[Proof of \ref{thm:main_ii}]
    Let $G_k$ be the graph where vertices are indexed by tuples $(u,v)$ with $u \in \{1, 2, 3\}$ and $v \in \{1, \dots, k\}$, and $(u_1,v_1)(u_1,v_2) \in E(G_k)$ 
    if $u_1u_2 \in E(K_3)$ and $v_1v_2 \in E(K_k)$. See \cref{fig:product} for an example representing~$G_5$.
    Alternatively, $G_k$ is the tensor product $K_3\times K_k$ of cliques of sizes $3$ and $k$ respectively, or the complement of the $3\times k$ Rook graph $K_3 \square K_k$. 

    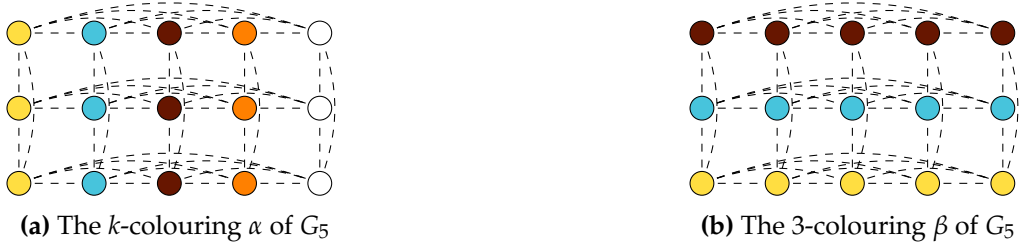
\begin{figure}[h!]
    \begin{subfigure}{.45\textwidth}
    \centering
    \begin{tikzpicture}
  
    \foreach \x in {1,2,3,4,5}{
    \foreach \y in {1,2,3}{
      \coordinate (A\x\y) at (\x,\y);
    }}

    \foreach \x in {1, ..., 5}{
    \foreach \y/\w in {1/2, 2/3}{
      \draw[dashed] (A\x\y) --(A\x\w);
    }
      \path[every node/.style={font=\sffamily\small}]
      (A\x1) edge[bend right=20, dashed] node {} (A\x3);
    }      

    \foreach \x in {1, 2, 3}{
    \foreach \y/\w in {1/2, 2/3, 3/4, 4/5}{
      \draw[dashed] (A\y\x) --(A\w\x);
    }
    \foreach \y/\w in {1/3, 1/4, 1/5, 2/4, 2/5, 3/5}{
      \path[every node/.style={font=\sffamily\small}]
      (A\y\x) edge[bend left=20, dashed] node {} (A\w\x);
    }}

    \foreach \x/\c in {1/\colora,2/\colorb,3/\colorc,4/\colord,5/white}{
    \foreach \y in {1,2,3}{
      \node[draw=black,fill=\c,circle,inner sep=1pt, minimum size = 1.8*\ray] at (A\x\y) {};
    }}
    \end{tikzpicture}
    \caption{The $k$-colouring $\alpha$ of $G_5$}
    \end{subfigure}
    \hfill
    \begin{subfigure}{.45\textwidth}
    \centering
    \begin{tikzpicture}
  
    \foreach \x in {1,2,3,4,5}{
    \foreach \y in {1,2,3}{
      \coordinate (A\x\y) at (\x,\y);
    }}

    \foreach \x in {1, ..., 5}{
    \foreach \y/\w in {1/2, 2/3}{
      \draw[dashed] (A\x\y) --(A\x\w);
    }
      \path[every node/.style={font=\sffamily\small}]
      (A\x1) edge[bend right=20, dashed] node {} (A\x3);
    }      

    \foreach \x in {1, 2, 3}{
    \foreach \y/\w in {1/2, 2/3, 3/4, 4/5}{
      \draw[dashed] (A\y\x) --(A\w\x);
    }
    \foreach \y/\w in {1/3, 1/4, 1/5, 2/4, 2/5, 3/5}{
      \path[every node/.style={font=\sffamily\small}]
      (A\y\x) edge[bend left=20, dashed] node {} (A\w\x);
    }}

    \foreach \y/\c in {1/\colora,2/\colorb,3/\colorc}{
    \foreach \x in {1,2,3,4,5}{
      \node[draw=black,fill=\c,circle,inner sep=1pt, minimum size = 1.8*\ray,] at (A\x\y) {};
    }}
    \end{tikzpicture}
    \caption{The 3-colouring $\beta$ of $G_5$}
    \end{subfigure}
    \caption{Colourings of $G_5$. Non-edges are depicted as dashed lines, while edges are omitted for clarity.}\label{fig:product}
    \end{figure}

    Note that $G_k$ is perfect \cite{ravindra1977Perfect}. Moreover, it is $2(k-1)$-regular and has clique number three. So we have that: $$k = \left\lceil \frac{\omega(G_k)+\Delta(G_k) +1}{2}\right\rceil -1.$$

    Moreover, $G_k$ has a frozen $k$-colouring $\alpha$ with $\alpha((u,v)) \coloneq v$ and a $3$-colouring $\beta$ with $\beta(u,v)) \coloneq u$, which is a fortiori a $k$-colouring, so $G_k$ is not $k$-recolourable. 
\end{proof}

This proves that \cref{thm:reed_odd_hole,thm:main} are tight up to one colour even for perfect graphs, in the sense that the only open case of \cref{conj:odd_hole_reed} is $k = f(G)$.

\section{Frozen colourings of general graphs}\label{sec:tight}
For connected graphs, the random construction of Bonamy, Heinrich, Legrand-Duchesne and Narboni~\cite{MR4660623}, improves the bound in \ref{thm:main_ii} by giving graphs that admit a frozen non-unique $\lceil\eta\omega + (1-\eta)(\Delta+1)\rceil$-colouring for all $\eta > 1/3$. 
This shows that $\eta^\star\leq 1/3$ for connected graphs, and we prove that this bound is optimal.
\tight*
\begin{proof}
  Let $\mathcal{G}$ be the class of all connected graphs. 
  As cliques have only
  one colouring up to colour permutation, $\eta^\star(\mathcal{G}) = \eta^\star(\mathcal{G}')$,
  where $\mathcal{G}' = \mathcal{G} \setminus \{K_t : t \ge 1\}$. Thus, our goal is to prove that $\eta^\star := \eta^\star(\mathcal{G}')\geq1/3$. 

\begin{claim}\label{lem:no_isolated}
    Let $\alpha_1$ be a frozen $k_1$-colouring of a graph $G_1 \in \mathcal{G}'$ with a colour class of size one. 
    Then at least one of the following holds:
    \begin{itemize}
    \item there exists a graph $G_2 \in \mathcal{G}'$ with a frozen $k_2$-colouring and all colour classes of size at least two such that $\eta_2 \le \eta_1$, where $\eta_i$ satisfies $k_i = \eta_i\omega(G_i) + (1-\eta_i)(\Delta(G_i) +1)$ for $i=1,2$,
    \item or $\eta_1 = 1$.
    \end{itemize}
\end{claim}
\begin{poc}
    Let $X$ be the set of vertices that are uniquely coloured in $\alpha_1$. 
    Let $G_2 =G_1 \setminus X$ and $\alpha_2$ be the $(k_1 - |X|)$-colouring induced by $\alpha_1$ on $G_2$.
    Since $\alpha_1$ is frozen, $X$ induces a clique in $G_1$ and $\alpha_2$ is frozen.
    As $G_1$ is not a clique, there is at least one colour class in $\alpha_1$ that has size at least two.
    If there is only one such colour class, then we have that $k_1 = \omega(G_1) < \Delta(G_1) + 1$, which implies that $\eta_1 = 1$. 
    Thus, we can assume that there are at least two colour classes of size at least two in $\alpha_1$. Since every pair of colour classes induces a connected graph in $G_1$, $G_2$ is connected and $\alpha_1|_{G_2}$ is frozen. 
    And since $G_2$ has a colour class of size at least two, $G_2$ is not a clique and is in $\mathcal{G}'$. 
    
    The vertices of $X$ dominate $G_1$, so $\Delta(G_1) =|G_1|-1$ and $\Delta(G_2) \le \Delta(G_1)-|X|$. On the other hand, $\omega(G_2) = \omega(G_1)- |X|$ and $k_2 = k_1 - |X|$. Hence, from $G_1$ to $G_2$, the number of colours of the frozen colouring and the clique number decreased exactly by $|X|$, while the maximum degree decreased by at least $|X|$, so $\eta_2\le\eta_1$.
\end{poc}

For all $G$ and $k$, denote $\eta^{(k)}(G)$ the infimum over all $\eta$ such that there exists a frozen non-unique $k$-colouring $\alpha$ of $G$, with $k = \eta\omega(G)+(1-\eta)(\Delta(G) +1)$ and such that all colour classes of $\alpha$ have size at least two. Denote also $\eta(G) = \inf \{\eta^{(k)}(G) : k \in \mathbb{N}\}$. By \cref{lem:no_isolated}, we have that $\eta^\star = \inf \{\eta(G) : G \in \mathcal{G}'\}$.

\begin{claim}\label{lem:optimal}
    Let $\alpha$ be a frozen $k$-colouring of $G \in \mathcal{G}'$ with no colour class of size one. Then $\Delta(G)\ge \frac32 (k-1)$.
    As a result, $\eta^{(k)}(G) > 1/3$.
\end{claim}
\begin{poc}
    Let $U$ be a colour class of minimal size. We have $n \ge 2k$ and by \cref{lem:avgdeg_frozen},
    \begin{align*}
    \Delta(G) \ge \deg(U)/|U| &\ge \frac{n-|U|+(k-1)(|U|-1)}{|U|}\\ 
    &\ge \frac32(k-1).
    \end{align*}

    So $k \le 2 \Delta(G)/3 +1$. We have $k =\eta^{(k)}(G)\omega (G)+ (1-\eta^{(k)}(G))(\Delta (G) + 1)$ and thus
    $$ \eta^{(k)}(G)= \frac{\Delta (G) +1-k}{\Delta (G) +1- \omega(G)} > \frac{\Delta (G) +1-k}{\Delta (G)}\ge \frac13.$$
    \end{poc}

This proves that $\eta(G)$ and $\eta^\star$ are at least $1/3$, and hence $\eta^\star=1/3$.
Moreover, note that $\eta^\star$ is unattained in $\mathcal{G}$, otherwise some graph $G \in \mathcal{G}$ would verify $\eta^{(k)}(G) = 1/3$.
As all frozen non-unique $k$-colourings of any graph $G$ must verify $k > \eta^\star\omega(G) + (1-\eta^\star)(\Delta +1)$, this concludes the proof of \cref{thm:no_frozen}.
\end{proof}

\section{Frozen colourings of triangle-free graphs}\label{sec:triangle_free}
The random construction of Bonamy, Heinrich Legrand-Duchesne and Narboni has
clique number $\Theta(\log k)$. We show here that $\eta^\star$ changes when
restricting to graphs of bounded clique number.
\cref{thm:triangle_free}, that we recall below, shows that
$\eta^\star(\{\text{triangle-free graphs}\}) = 4/9$.

\trianglefree*

We prove the two parts of \cref{thm:triangle_free} separately.

\subsection*{Existence of frozen colourings of triangle-free graphs with $\eta > 4/9$}
To prove \ref{thm:triangle_free_i} and show that $\eta^\star(\{\text{triangle-free graphs}\}) \le  4/9$, we construct triangle-free graphs with a frozen non-unique $k$-colouring and small maximum degree:

\begin{lemma}
    \label{lem:frozen_triangle_free}
    For any~$k$ there exists a triangle-free
    graph $G_k$ with maximum degree at most $(9k-1)/5$, that admits a frozen $k$-colouring
    and is not $k$-recolourable.
\end{lemma}

\ref{thm:triangle_free_i} follows directly from \cref{lem:frozen_triangle_free}:

\begin{proof}[Proof of \ref{thm:triangle_free_i} assuming \cref{lem:frozen_triangle_free}]
Let $\eta$ such that the graph $G_k$ in
\cref{lem:frozen_triangle_free} verifies $k = 2\eta + (1-\eta)(\Delta(G_k)+1)$. We have $$\eta = \frac{\Delta(G_k)+1-k}{\Delta(G_k)-1} \leq \frac{4(k+1)}{9k-6}.$$ As
this inequality holds for any $k$, this proves \ref{thm:triangle_free_i} and $\eta^\star(\{\text{triangle-free graphs}\}) \le 4/9$.
\end{proof}

\begin{proof}[Proof of \cref{lem:frozen_triangle_free}]
  Let $k \ge 0$. We first build a $(2k-2)$-regular triangle-free graph $H_k$
  with a frozen non-unique $k$-colouring. Then, we will consider a subgraph $G_k$
  of $H_k$ to reduce its maximum degree while preserving the other properties
  of $H_k$.

  Consider the graph $H_k$ on $5k$ vertices defined as follows (the
  construction is illustrated on \cref{fig:optimal_triangle_free}). Partition
  the vertices into $k$ colour classes $V_1, \dots, V_k$ of size $5$ and number
  the vertices from 1 to 5 within each colour class.
  For any integers $i<j$, for any $x \in [5]$, connect the vertex labelled $x$
  in $V_i$ to the vertices labelled $x + 2 \mod 5$ and $x + 3 \mod 5$ in
  $V_j$. This results in a $(2k-2)$-regular graph $H_k$ with a frozen
  $k$-colouring $\alpha$ whose colour classes are the sets $V_i$.
  \begin{figure}[h!]
\centering
\begin{tikzpicture}[scale = 1.3]
  \pgfmathsetmacro{\k}{6}
  \pgfmathsetmacro{\km}{\k-1}
  \pgfmathsetmacro{\n}{5}

  \foreach \i in {1,...,\k}{
    \foreach \j in {1,...,\n}{
      \coordinate (a\i\j) at (\i*360/\k + \j*360/\k/\n:5cm);
    }
  }

  \foreach \a in {1,...,\km}{
    \pgfmathsetmacro{\i}{int(\a +1)}
    \foreach \b in {\i,...,\k}{
      \foreach \c in {1,...,\n}{
        \pgfmathsetmacro{\x}{int(Mod(\c+1,\n)+1}
        \pgfmathsetmacro{\y}{int(Mod(\c+2,\n)+1}
        \draw (a\b\x) -- (a\a\c) -- (a\b\y);          
      }
    }
  }

  \foreach \j in {1,...,\n}{
    \node[draw=black,fill=\colora,circle,inner sep = 1pt, minimum size = 2*\ray] at (a1\j) {};
    \node[draw=black,fill=\colorb,circle,inner sep = 1pt, minimum size = 2*\ray] at (a2\j) {};
    \node[draw=black,fill=\colorc,circle,inner sep = 1pt, minimum size = 2*\ray] at (a3\j) {};
    \node[draw=black,fill=\colord,circle,inner sep = 1pt, minimum size = 2*\ray] at (a4\j) {};
    \node[draw=black,fill=\colore,circle,inner sep = 1pt, minimum size = 2*\ray] at (a5\j) {};
    \node[draw=black,fill=\colorf,circle,inner sep = 1pt, minimum size = 2*\ray] at (a6\j) {};
  }  

 \begin{scope}[rotate= 360/\k + 360/\k/\n]
    \braceme{5.2cm}{-5}{360/\k - 360/\k/\n+5}{br1}{}
 \end{scope}     
 \begin{scope}[rotate= 2*360/\k + 360/\k/\n]
    \braceme{5.2cm}{-5}{360/\k - 360/\k/\n+5}{br2}{}
 \end{scope}     
 \begin{scope}[rotate= 360/\k/\n]
    \braceme{5.2cm}{-5}{360/\k - 360/\k/\n+5}{br3}{}
 \end{scope}     
 \node at (360/\k + 3*360/\k/\n:5.5cm) {$V_1$};
 \node at (2*360/\k + 3*360/\k/\n:5.5cm) {$V_2$};
 \node at (3*360/\k/\n:5.5cm) {$V_k$};

\end{tikzpicture}
\caption{The graph $H_6$. Within each set $V_i$ the vertices are labelled in anti-clockwise order.}\label{fig:optimal_triangle_free} 
\end{figure}
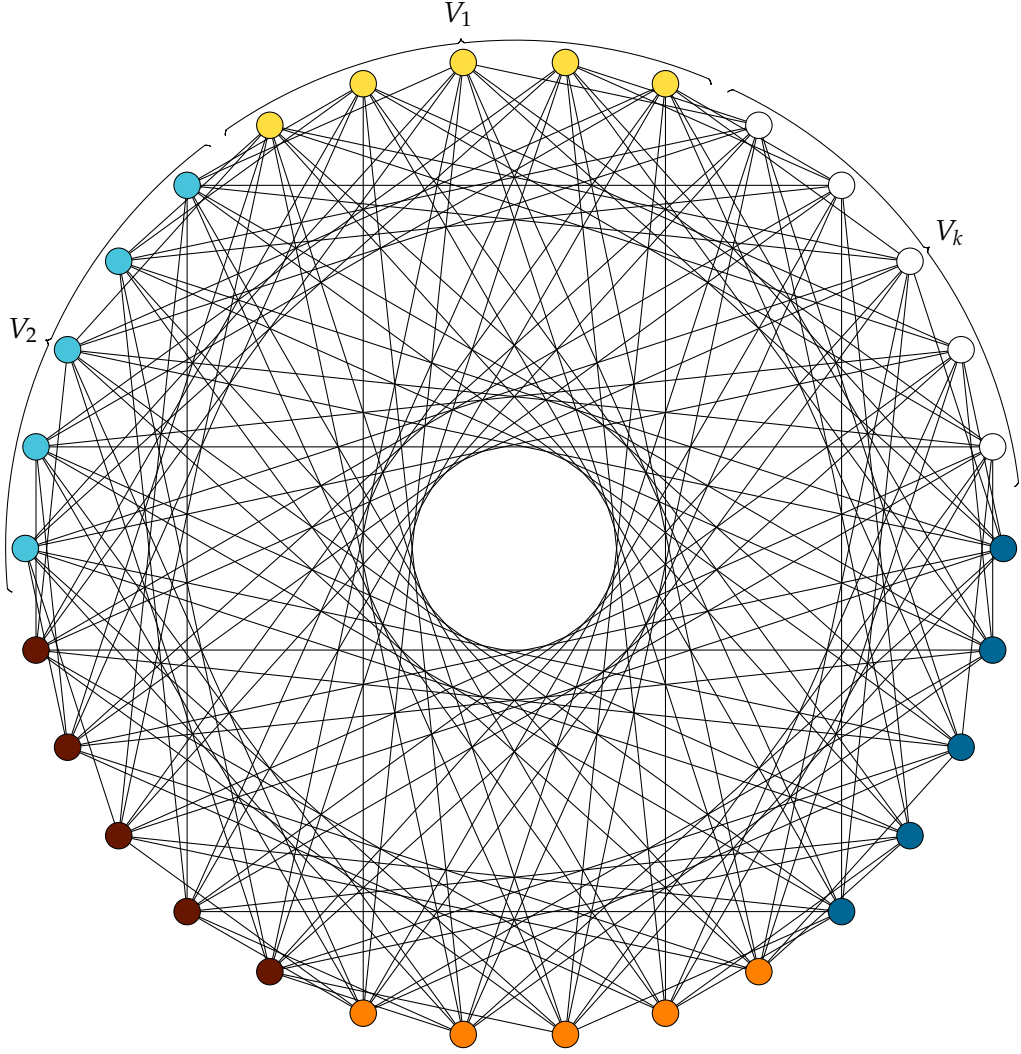  
 
  We first argue that $H_k$ is triangle-free. Indeed, any triangle would have to
  use vertices in three different colour classes, say $V_i$, $V_j$ and $V_\ell$,
  with $i <j <\ell$. Denote $x$ the label of the vertex of $V_i$ in this
  triangle. By construction, this would imply the following relation
  $x + a_{ij} + a_{j\ell} = x + a_{i\ell} \mod 5$, where
  $a_{ij}, a_{j\ell}, a_{i\ell} \in \{2,3\}$. In other words,
  $0 \in \{1, \dots, 4\} \mod 5$, which is a contradiction.

  Let~$G_k$ be the graph obtained from~$H_k$ as follows. For each $i<j$, we delete the edge between the vertex labelled~$x=j-i \mod 5$ in~$V_i$ and the vertex labelled~$y=x+2 \mod 5$ in~$V_j$. Note that~$\alpha$ remains a frozen $k$-colouring in~$G_k$ as any pair of colours induces a path on ten vertices. Moreover, the degree of any vertex in~$V_i$ is at most \[ 2(k-1) - \left\lfloor \frac{i-1}{5} \right\rfloor - \left\lfloor \frac{k-i}{5} \right\rfloor \leq 2(k-1) - \frac{i-5}{5} - \frac{k-i-4}{5} = \frac{9k-1}{5}~.\]

  It remains only to prove that $G_k$ admits a different $k$-colouring, which
  follows from the fact that $H_k[V_1 \cup V_2 \cup V_3]$ admits another
  3-colouring (see \cref{fig:other_triangle_free}).

 \begin{figure}[h!]
\centering
\begin{tikzpicture}[scale = 1]
  \pgfmathsetmacro{\k}{3}
  \pgfmathsetmacro{\km}{\k-1}
  \pgfmathsetmacro{\n}{5}

  \foreach \i in {1,...,\k}{
    \foreach \j in {1,...,\n}{
      \coordinate (a\i\j) at (\i*360/\k + \j*360/\k/\n:4cm);
    }
  }

  \foreach \a in {1,...,\km}{
    \pgfmathsetmacro{\i}{int(\a +1)}
    \foreach \b in {\i,...,\k}{
      \foreach \c in {1,...,\n}{
        \pgfmathsetmacro{\x}{int(Mod(\c+1,\n)+1}
        \pgfmathsetmacro{\y}{int(Mod(\c+2,\n)+1}
        \draw (a\b\x) -- (a\a\c) -- (a\b\y);          
      }
    }
  }

  \foreach \a in {13,23,33,12,22,32}{
    \node[draw=black,fill=\colora,circle,inner sep = 1pt, minimum size = 2*\ray] at (a\a) {};
  }  
  \foreach \a in {11,21,31}{
    \node[draw=black,fill=\colorb,circle,inner sep = 1pt, minimum size = 2*\ray] at (a\a) {};
  }  
  \foreach \a in {14,24,34,15,25,35}{
    \node[draw=black,fill=\colorc,circle,inner sep = 1pt, minimum size = 2*\ray] at (a\a) {};
  }  
 \begin{scope}[rotate= 360/\k + 360/\k/\n]
    \braceme{4.5cm}{-5}{360/\k - 360/\k/\n+5}{br1}{}
 \end{scope}     
 \begin{scope}[rotate= 2*360/\k + 360/\k/\n]
    \braceme{4.5cm}{-5}{360/\k - 360/\k/\n+5}{br2}{}
 \end{scope}     
 \begin{scope}[rotate= 360/\k/\n]
    \braceme{4.5cm}{-5}{360/\k - 360/\k/\n+5}{br3}{}
 \end{scope}     
 \node at (360/\k + 3*360/\k/\n:4.8cm) {$V_1$};
 \node at (2*360/\k + 3*360/\k/\n:4.8cm) {$V_2$};
 \node at (3*360/\k/\n:4.8cm) {$V_3$};

\end{tikzpicture}
\caption{Another 3-colouring of $H_k[V_1 \cup V_2 \cup V_3]$. Within each set $V_i$ the vertices are labelled in anti-clockwise order.}\label{fig:other_triangle_free} 
\end{figure}
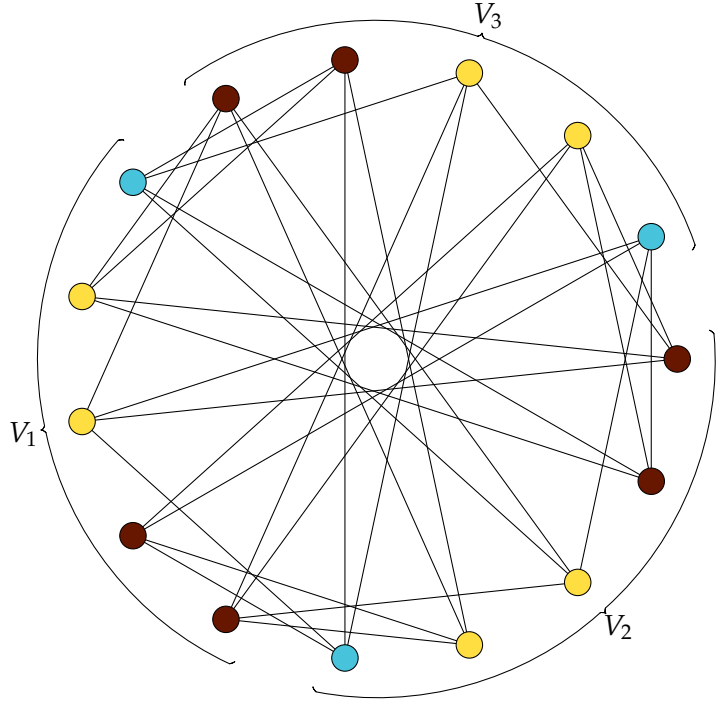
\end{proof}    

Note that this construction is similar to the construction of Bonamy, Heinrich, Legrand-Duchesne and Narboni: they also build a $(2k-2)$-regular graph $H_k$ with a frozen $k$-colouring such that all its colour classes have size two (instead of five in our case) and any two pair of colours span a cycle. Then, they remove independently at random an edge from each of those cycles to decrease the maximum degree. In our case, this random edge-removal procedure would also concentrate the degrees around $9(k-1)/5$ using Chernoff bound. This yields a graph $G_k$ with maximum degree at most $\delta k$ for any $\delta > 9/5$ and large enough $k$, thereby proving a slightly weaker version of \cref{lem:frozen_triangle_free} that still implies \ref{thm:triangle_free_ii}. The deterministic edge-removal procedure we described instead is less flexible but is constructive and has maximum degree at most $(9k-1)/5$. 

\subsection*{Non-existence of frozen colourings of triangle-free graphs with $\eta \le 4/9$}
\begin{proof}[Proof of \ref{thm:triangle_free_ii}]
  Let $G$ be a triangle-free graph with a frozen $k$-colouring $\alpha$ that is
  non-unique (hence $G$ is not a clique), with
  $k =\lceil2\eta + (1-\eta)(\Delta +1)\rceil$ for some $\eta \le
  4/9$. Here again, one can assume without loss of generality that
  $k \ge 2\eta + (1-\eta)(\Delta +1)$. Therefore, $\Delta(G) \le (k - 1-\eta)/(1-\eta) < 9(k-1)/5$. Furthermore, since $G$ is not a clique and $\alpha$ is frozen, we have $\Delta(G)+1 \ge 3$ and since $\eta \le 4/9$, note that $k \ge 3$.
  
  We first bound the number of colour classes of size $m$ in $\alpha$, for all
  $m \le 4$.
  \begin{claim}\label{cl:number_colour_classes}
    The colouring $\alpha$ has no colour class of size one or two, at most two colour classes of size three and at most three colour classes of size four.  
  \end{claim}
  \begin{poc}
    We first prove that no colour class of $G$ is dominated by a vertex. Assume
    otherwise and let $u$ be a vertex dominating a colour class $U$. Let $v$ be a neighbour
    of $u$ not in $U$ (such a vertex exists because $k \ge 3$) and $w$ be a
    neighbour of $v$ in $U$. The vertices $u$, $v$ and $w$ form a triangle.
    In particular, $\alpha$ has no colour class of size one or two as it would be dominated by some vertex in any other colour class.

    Suppose that $\alpha$ has three colour classes $V_1$, $V_2$ and $V_3$ of size three. Since no vertex dominates a colour class, each pair of colour classes spans a path on 6 vertices, say $u_1, w_2, v_1, v_2, w_1, u_2$ with $u_i$, $v_i$ and $w_i$ in $V_i$ for $i \in \{1,2\}$. For each $i \le 2$, $X_i = \{u_i,w_i\}$ dominates $V_{3-i}$, so $u_i$ and $w_i$ cannot have a common neighbour $x$ in $V_3$, otherwise any neighbour of $x$ in $V_{3-i}$ forms a triangle with $x$ and one of the vertices in $X_i$. Thus, for each $i\le 2$, $u_i$ or $w_i$ has degree one in $V_3$. In particular $\deg(v_i,V_3) = 2$ for each $i \le 2$, so $v_1$ and $v_2$ have a common neighbour in $V_3$, but as $v_1$ and $v_2$ are adjacent, this creates a triangle.
    
    Suppose now that $\alpha$ has four colour classes of size four. One can assume without loss of generality that each pair of them spans a tree. As no vertex dominates another colour class, the only possible adjacencies between any two of these colour classes are those depicted on \cref{fig:adjacencies}. To avoid a tedious case analysis, we computer checked that it is not possible to build a frozen four-colouring using those\footnote{See \href{https://github.com/johnlepoulpe/Listing-frozen-4-colourings/blob/main/listing_frozen_triangle_free.ipynb}{here} for the code that was used.}.
    
    \begin{figure}[h!]
  \centering
    \begin{tikzpicture}
        \foreach \x in {0,...,3}{
            \node[circle,draw=black,inner sep =1pt, minimum size=  \ray,fill=\colora]  (a\x) at (0,\x) {};
            \node[circle,draw=black,inner sep =1pt, minimum size=  \ray,fill=\colorb]  (b\x) at (1.5,\x) {};
        }
        \draw (a0) -- (b0) -- (a1) -- (b1) -- (a2) -- (b2) -- (a3) -- (b3);
    \end{tikzpicture}
    \hfill
       \begin{tikzpicture}
        \foreach \x in {0,...,3}{
            \node[circle,draw=black,inner sep =1pt, minimum size=  \ray,fill=\colora]  (a\x) at (0,\x) {};
            \node[circle,draw=black,inner sep =1pt, minimum size=  \ray,fill=\colorb]  (b\x) at (1.5,\x) {};
        }
        \draw (b0) -- (a0) -- (b1) -- (a1) (a0) -- (b2) (b2) -- (a2) -- (b3) -- (a3);
    \end{tikzpicture}
    \hfill
    \begin{tikzpicture}
        \foreach \x in {0,...,3}{
            \node[circle,draw=black,inner sep =1pt, minimum size=  \ray,fill=\colora]  (a\x) at (0,\x) {};
            \node[circle,draw=black,inner sep =1pt, minimum size=  \ray,fill=\colorb]  (b\x) at (1.5,\x) {};
        }
        \draw (a1) -- (b0) -- (a0) -- (b1) -- (a2) (a0) -- (b2) -- (a3) -- (b3);
    \end{tikzpicture}
    \hfill
    \begin{tikzpicture}
        \foreach \x in {0,...,3}{
            \node[circle,draw=black,inner sep =1pt, minimum size=  \ray,fill=\colora]  (a\x) at (0,\x) {};
            \node[circle,draw=black,inner sep =1pt, minimum size=  \ray,fill=\colorb]  (b\x) at (1.5,\x) {};
        }
        \draw (a1) -- (b0) -- (a0) -- (b1)  (a0) -- (b2) (a3) -- (b3) -- (a2) -- (b0) ;
    \end{tikzpicture}
    \hfill
    \begin{tikzpicture}
        \foreach \x in {0,...,3}{
            \node[circle,draw=black,inner sep =1pt, minimum size=  \ray,fill=\colora]  (a\x) at (0,\x) {};
            \node[circle,draw=black,inner sep =1pt, minimum size=  \ray,fill=\colorb]  (b\x) at (1.5,\x) {};
        }
        \draw (a1) -- (b0) -- (a0) -- (b1)  (a0) -- (b2) -- (a3)  (b3) -- (a2) -- (b0)  ;
    \end{tikzpicture}
     
  \caption{All the possible adjacencies between two colour classes of size four in $\alpha$, up to isomorphism.}
  \label{fig:adjacencies}
\end{figure}
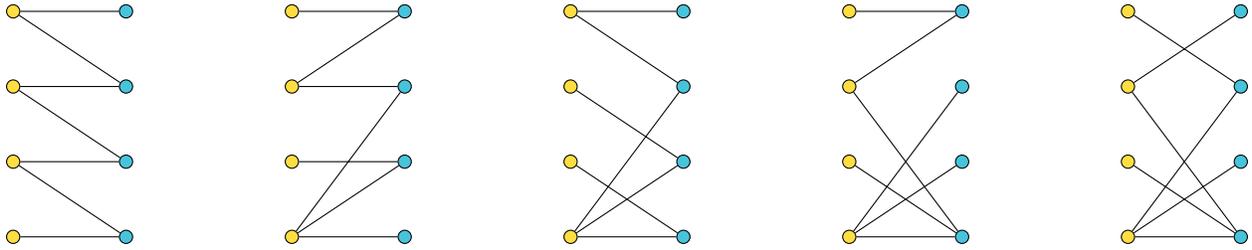
  \end{poc}

  We can now improve on \cref{cl:number_colour_classes} and prove that in fact
  all colour classes have size at least five.
  \begin{claim}\label{cl:size_colour_classes}
    The colouring $\alpha$ has no colour class of size less than five.
  \end{claim}
    \begin{poc}
      Let $U$ be one of the smallest colour classes.
      By ~\cref{cl:number_colour_classes}, we have $|U| \ge 3$. We have $\Delta(G) \ge \lceil\deg(U)/|U|\rceil$
      and by \cref{lem:avgdeg_frozen}
      $$\frac{\deg(U)}{|U|}\ge \frac{n-|U| +
        (k-1)(|U|-1)}{|U|} = \frac{n + k(|U|-1) +1 -2|U|}{|U|}$$

      We will first determine a lower bound on $\deg(U)/|U|$, before raising a contradiction.
      If $|U| = 3$, then by \cref{cl:number_colour_classes} we have $n \ge
      6+4(k-2)$ and thus $\deg(U)/|U| \ge (6k-7)/3$.
      If $|U| = 4$, then by \cref{cl:number_colour_classes} we have $n \ge
      12+5(k-3)$ and thus $\deg(U)/|U| \ge (8k-10)/4$.

      In any of theses cases, $\Delta(G) \ge \lceil\deg(U)/|U|\rceil \ge 2k
      -2$. This contradicts $\Delta(G) < 9(k-1)/5$, as $k \ge 3$.
    \end{poc}

    Let $U$ be one of the smallest colour classes. So $n \ge k|U|$ and by
    \cref{cl:size_colour_classes} $|U| \ge 5$. By
    \cref{lem:avgdeg_frozen},
    \begin{align*}
      \frac{\deg(U)}{|U|} &\ge \frac{n + k(|U|-1) +1 -2|U|}{|U|}\\
                  &\ge 2(k-1) - \frac{k-1}{|U|}   \\
                  &\ge \frac95(k-1).
    \end{align*}
    Hence $\Delta(G) \ge \frac95(k-1)$ which is a contradiction.
\end{proof}

\section*{Acknowledgments}
The authors would like to thank Tom\'a\v{s} Kaiser and Marthe Bonamy for the fruitful discussions in the early stages of this project. This research was conducted during the workshop Early Career Researchers in Combinatorics (ECRiC24), held at ICMS in Edinburgh.\\
The first author was supported by ANR project GrR (ANR-18-CE40-0032). \\
The second author was partially supported by the Polish National Science Centre under grant no. 2019/34/E/ST6/00443.\\
The third author was supported by the Warwick Mathematics Institute Centre for Doctoral Training.\\
The fourth author was supported by DFG, project number 546892829.

\bibliographystyle{style_edited}
\bibliography{references}
\end{document}